\documentclass[11pt]{article} 
\usepackage[top=1in,right=1in,left=1in,bottom=1in]{geometry}

\usepackage{amsmath}
\usepackage{amssymb}
\usepackage{amsthm} 
\usepackage{graphicx}
\usepackage{subfigure}
\usepackage{pgf}
\usepackage{xcolor}

\newtheorem{lemma}{Lemma}
\newtheorem{theorem}{Theorem}
\newtheorem{proposition}{Proposition}

\newcommand{\beq}{\begin{equation}}
\newcommand{\eeq}{\end{equation}}
\newcommand{\barr}{\begin{array}}
\newcommand{\earr}{\end{array}}
\newcommand{\bea}{\begin{eqnarray}}
\newcommand{\eea}{\end{eqnarray}}

\newcommand{\innerproduct}[2]{\langle #1, #2 \rangle}

\newcommand{\pd}[2]{\frac{\partial #1}{\partial #2}}
\newcommand{\pdt}{\frac{\partial}{\partial t}}

\newcommand{\mtwo}[4]{\begin{pmatrix} #1 & #2 \\ #3 & #4 \end{pmatrix}}
\newcommand{\vtwo}[2]{\begin{pmatrix} #1 \\ #2 \end{pmatrix}}

\newcommand{\pmat}[1]{\begin{pmatrix} #1 \end{pmatrix}}

\newcommand{\eps}{\epsilon}

\begin{document}

\title{Energy-Superconvergent Explicit Runge--Kutta Time Discretizations
}


\author{Jinjie Liu \footnote{Division of Physics, Engineering, Mathematics, and Computer Science, Delaware State University, Dover, DE 19901 USA. jliu@desu.edu.} 
	~~~and~~~
	Moysey Brio \footnote{Department of Mathematics, The University of Arizona, Tucson, AZ 85721 USA. brio@math.arizona.edu.}
}


\maketitle

\begin{abstract}
This paper investigates the energy conservation properties of explicit Runge--Kutta (RK) time discretizations for autonomous skew-symmetric  systems. For linear problems, we present a general framework for constructing RK methods in which the energy-accuracy order significantly exceeds the number of stages. Specifically, for an $s$-stage, $p$-th order RK method (where $p$ is even), we prove that the energy accuracy can reach up to order $2s-p+1$. Utilizing this framework, we derive several energy-superconvergent methods, 
including five- to seven-stage algorithms with energy accuracy up to the eleventh order, 
and establish their corresponding strong stability criteria. 
The methods are validated on a range of benchmark problems, 
including harmonic oscillators, integro-differential equations in peridynamics, and the Maxwell equations.
	
Furthermore, we extend the energy-superconvergent framework to autonomous nonlinear systems with amplitude-dependent frequencies. 
By deriving fifth-order energy conditions for three-stage, second-order methods, we develop the RK325 algorithm. 
The performance of RK325 is demonstrated for a broad range of problems, including Euler's equations for rigid body dynamics, the nonlinear Schr\"odinger equation, the Korteweg--de Vries (KdV) equation, Burgers' equation, and the Landau--Lifshitz equation. 
Additionally, we develop four-stage, second-order methods (RK427) and five-stage, fourth-order methods (RK547), all of which achieve seventh-order energy accuracy for the cubic nonlinear case.
Finally, the performance of RK547 method is illustrated using the nonlinear Maxwell--Kerr system.


\end{abstract}

\section{Introduction}
\label{intro}

Runge--Kutta (RK) algorithms are widely used numerical methods for solving systems of ordinary differential equations (ODEs) and ODE systems obtained from semi-discretized partial differential equations (PDEs). 
Besides stability and solution convergence rates, energy accuracy is an important criterion for selecting RK methods, especially for energy conserving systems. Generally, only implicit RK methods can preserve the energy \cite{cooper1987stability,sanz1988runge}. 
However, implicit methods require the solution of a large system of equations at every iteration. 
In contrast, explicit RK methods are computationally efficient, but they are generally not energy-preserving. 
Nevertheless, energy-based analysis has been widely used to study stability properties and to construct strong stability preserving (SSP) methods  \cite{levy1998semidiscrete,gottlieb1998total,gottlieb2001strong,tadmor2002semidiscrete,sun2017stability,sun2019strong,cockburn2001runge}.

Energy-conserving time integrators remain a very active area of research, 
and recent developments include 
the averaged vector field (AVF) method \cite{quispel2008new,celledoni2010energy,hairer2010energy}, 
the relaxation RK methods \cite{ketcheson2019relaxation,ranocha2020relaxation},	
the scalar auxiliary variable (SAV) approach \cite{shen2018scalar,antoine2021scalar}, 
the energy-conserving SAV discontinuous Galerkin method for the nonlinear Dirac equation \cite{yang2022energy}, 
the successive multi-stage method for the linear wave equation \cite{shin2022energy,shin2023energy}, 
and the energy-preserving relaxed implicit-explicit (IMEX) RK method \cite{cui2025high}.

In this paper, we investigate the energy accuracy of explicit RK methods for skew-symmetric autonomous systems. 
Typical applications include Hamiltonian systems, semi-discretized Maxwell's equations, 
nonlinear Schr\"odinger equations, the Korteweg--de Vries (KdV) equation, Burgers' equation, Landau--Lifshitz equations, 
and peridynamic nonlocal PDEs \cite{Silling2000}.

The energy accuracy of linear autonomous systems has been studied in \cite{sun2019strong}, where it was shown that the energy order is at least $p+1$ for a method of even order $p$. 
For an $s$-stage method, the discrete energy equation is a polynomial of degree $2s$ in the time step $h$. Hence, the highest possible energy accuracy is of order $2s-1$, provided that all lower-order terms vanish. In general, imposing these conditions leads to an overdetermined system for the RK coefficients. 

For skew-symmetric systems, however, all odd-order terms vanish automatically, thereby reducing the constraints and making the resulting system solvable.
Solving the resulting system for the RK coefficients yields a class of energy-superconvergent RK methods in which the order of energy accuracy exceeds the number of stages.
For an $s$-stage method of even order $p$, the energy accuracy can reach the order of $2s-p+1$. For example, we construct a seven-stage fourth-order method with an energy order of 11. In addition, we derive a set of strongly stable fourth-order RK methods and their corresponding stability criteria. 

Note that SSP-RK methods \cite{gottlieb2001strong} have been developed for general linear and nonlinear systems. 
In contrast, the proposed methods are specifically designed for skew-symmetric systems, 
with an emphasis on achieving higher-order energy accuracy.
The proposed energy-superconvergent methods are tested using several one-dimensional examples, including second-order ODEs for harmonic oscillators, integro-differential equations for linear peridynamic models, and Maxwell's equations of electrodynamics. 

Furthermore, we extend the energy-superconvergent framework to autonomous nonlinear systems with amplitude-dependent frequencies.
First, we derive the order conditions for a three-stage, second-order RK method to achieve fifth-order energy accuracy. 
The resulting three-stage RK solver is applied to several nonlinear problems,
including Euler's equations for rigid body dynamics, the nonlinear Schr\"odinger equation, the KdV equation, Burgers' equation, and the Landau--Lifshitz equation. 
Second, we develop four- and five-stage methods achieving seventh-order energy accuracy for the cubic nonlinear case. 
Finally, the five-stage method is employed to solve the nonlinear Maxwell--Kerr system that models the electromagnetic wave propagation in third-order nonlinear media.

The paper is organized as follows. In Section \ref{sec:ESC}, we present the energy accuracy analysis that leads to the proposed energy-superconvergent methods. 
The extension to the nonlinear case is discussed in Section \ref{sec:nonlinear}.
Linear and nonlinear examples are presented in Sections \ref{sec:examples} and \ref{sec:nonlinear_examples}, respectively.
Conclusion and future work are discussed in Section \ref{sec:conc}.

\section{Energy-superconvergent time discretization}
\label{sec:ESC}

We consider the autonomous linear ordinary differential equation (ODE) system:
\beq
\frac{d}{dt} u = L u, \label{eq:de}
\eeq
where $L$ is skew-adjoint with respect to a symmetric and positive definite matrix $H$, i.e., $L^\top H + HL = 0$.  
In particular, when $H$ is an identity matrix, $L$ is skew-symmetric. 
The total energy (Hamiltonian) of this system is given by: 
\beq
\mathcal{E} = \frac{1}{2} ||u||^2_H = \frac{1}{2} \innerproduct{u}{u}_H. 
\eeq
Here, $\innerproduct{\cdot}{\cdot}_H$ represents the $H$-inner product, defined by $\innerproduct{x}{y}_H=\innerproduct{x}{Hy}$ where $\innerproduct{\cdot}{\cdot}$ is the Euclidean inner product. 
$||\cdot||_H$ is the corresponding $H$-norm. 
Throughout the rest of the paper, we drop the subscript $H$ and simply write $||\cdot||=||\cdot||_H$.

A general $s$-stage RK time discretization for the system (\ref{eq:de}) can be written as: 
\beq
u_{n+1} = G_s u_n, \label{eq:RK}
\eeq
where 
\beq 
G_s = \sum_{k=0}^s {a_k (hL)^k},~ a_0 = 1, ~ a_s\ne 0. \label{eq:G}
\eeq
Here, $h$ is the time step and $\{a_k\}_{k=0}^s$ are method dependent coefficients.
The method is of order $p$ if the first $p+1$ terms in $G_s$ coincide with the $p$-th Taylor polynomial of $e^{hL}$. 
 
Since $L^\top H + HL = 0$, we have (\cite{sun2019strong}, Corollary 2.2)
\beq
\innerproduct{L^iu}{L^ju}_H=\left\{
	\begin{array}{ll}
	(-1)^{(j-i)/2}||L^{(i+j)/2} u||^2, & i+j~\text{is even}, \\
	0, 							 & i+j~\text{is odd}, 
	\end{array} \right.
\eeq
and the energy equation
\beq 
\mathcal{E}_{n+1} = \mathcal{E}_n + \frac{1}{2} \sum_{k=1}^{s} b_k h^{2k} ||L^ku_n||^2, \label{eq:energy}
\eeq
where 
\beq  \label{eq:bk}
b_k = \sum_{i=\text{max}(0,2k-s)}^{\text{min}(2k,s)} (-1)^{k+i} a_i a_{2k-i} 
	= a_k^2 + 2 \sum_{i=1}^{\text{min}(k,s-k)}(-1)^ia_{k-i}a_{k+i}, 
\eeq
and $\mathcal{E}_n$ and $\mathcal{E}_{n+1}$ represent the energy at two consecutive times. 
Let $m$ denote the leading index of Equation (\ref{eq:energy}), i.e., $b_{m} \ne 0$ and $b_{k} = 0$ for all $1\le k < m$. 
Then the leading coefficient is $b_m$ and the order of energy accuracy is $r = 2m-1$. 
In this paper, we use $s$, $p$, and $r$ to represent the number of stages, the order of the solution accuracy, and the order of energy accuracy, respectively. 
An $s$-stage, $p$-th order method with $r$-th order of energy accuracy is denoted by RK($s$,$p$,$r$).

Proposition 4.7 in \cite{sun2019strong} reveals that $r=p$ if $p$ is odd, and $r \ge p+1$ if $p$ is even. 
When $p$ is even and $p=s$, using the binomial theorem 
\beq
\sum_{i=0}^{k}  \frac{(-1)^{i}}{i! (k-i)!} = 0, 
\eeq
we can calculate the leading coefficient as follows: 
\bea
b_{s/2+1} 
		    &=& (-1)^{s/2+1} \sum_{i=2}^{s} (-1)^i \cdot \frac{1}{i!} \cdot \frac{1}{(s+2-i)!}, \\
		    &=& (-1)^{s/2+1}\left( \sum_{i=0}^{s+2}  \frac{(-1)^i}{i! (s+2-i)!} - \frac{-2}{(s+1)!} - \frac{2}{(s+2)!} \right), \\
		    &=& (-1)^{s/2+1} \left( \frac{2}{(s+1)!}-\frac{2}{(s+2)!} \right) \ne 0, 
\eea
so $r = s+1$. 
This result is summarized in the following proposition.
\begin{proposition} \label{prop:p=s}
For an $s$-stage Runge--Kutta method of order $p=s$ applied to system (\ref{eq:de}), the order of energy accuracy is $r = 2\lfloor s/2 \rfloor + 1$, i.e., $r=s$ if $s$ is odd and $r=s+1$ if $s$ is even.
\end{proposition}

For instance, if $p=s=4$, then $a_1=1$, $a_2=1/2$, $a_3=1/6$, $a_4=1/24$, $b_1=b_2=0$, $b_3 = -1/72$, and $b_4 = 1/576$. 
Therefore, 
\beq
\mathcal{E} _{n+1} =\mathcal{E} _n - \frac{1}{72} h^6 ||L^3 u_n||^2 + \frac{1}{576} h^8 ||L^4 u_n||^2 = \mathcal{E} _n + \mathcal{O}(h^6), 
\eeq
and therefore $r=5$. In \cite{sun2017stability}, Corollary 2.1 proves that this four-stage fourth-order RK method is strongly stable if $h||L||\le 2\sqrt{2}$, where $||L||$ is the matrix norm: $||L||=\text{sup}_{||v||=1}||Lv||$. A more general result is given in the following lemma.

\begin{lemma} \label{lemma1}
An $s$-stage RK method applied to system (\ref{eq:de}), with a leading index $m=s-1$, is strongly stable if: \\
(i) $b_{s-1} < 0$ and \\
(ii) 
\beq \label{eq:stability} 
h||L|| \le \sqrt{\frac{|b_{s-1}|}{b_s}}.
\eeq
\end{lemma}
\begin{proof}
Because of condition (i), the energy equation (\ref{eq:energy}) becomes: 
\beq 
\mathcal{E}_{n+1} = \mathcal{E}_n + \frac{1}{2} b_{s-1} h^{2s-2} ||L^{s-1} u_n||^2 +  \frac{1}{2} b_{s} h^{2s} ||L^s u_n||^2,
\eeq
where $b_s=a_s^2>0$. Condition (ii) and (iii) imply that: 
\beq
b_s h^2||L||^2 \le - b_{s-1},  
\eeq
and thus 
\beq
\mathcal{E}_{n+1} - \mathcal{E}_n \le \frac{1}{2} ( b_{s-1} + b_s h^2||L||^2 ) h^{2s-2} ||L^{s-1} u_n||^2 \le 0.
\eeq
Therefore, the method is strongly stable. 
\end{proof}

When $p < s$ and $p$ is an even number, there are $s-p$ free parameters: $\{a_k\}_{k=p+1}^{s}$. If they form a solution to the system of the same number of equations $b_{k}=0$ for $p/2+1 \le k \le s-p/2$, then the leading index is $s-p/2+1$ and the order of energy accuracy is $r=2s-p+1$. Therefore, we have the following proposition. 
\begin{proposition} \label{prop:main}
For an $s$-stage RK method of order $p$ applied to system (\ref{eq:de}), if $p$ is even, then $p+1 \le r \le 2s-p+1$.
\end{proposition}

Next, we derive RK methods with energy order $r$ larger than $p+1$, including the case when $r$ reaches its upper bound $2s-p+1$.
When $p$ is even and $s=p+1$ or $s=p+2$, we can derive the expressions of the coefficients $\{a_k\}$ using the following two propositions.
\begin{proposition} \label{prop:p=s-1}
If $s>1$ is odd, and the coefficients of an $s$-stage method applied to the system (\ref{eq:de}) satisfy the following conditions:
\bea
a_k &=& \frac{1}{k!}, ~ k =1,2,...,s-1, 		\label{eq:p2_ak}\\
a_s &=& \frac{1}{s!} - \frac{1}{(s+1)!},  	\label{eq:p2_as}
\eea
then $p = s-1$ and $r = s+2$. 
\end{proposition}

\begin{proposition} \label{prop:p=s-2}
If $s>4$ is even, and the coefficients of an $s$-stage method applied to the system (\ref{eq:de}) satisfy the following conditions:
\bea
a_k 		&=& \frac{1}{k!}, ~ k =1,2,...,s-2,\\
a_{s-1} 	&=&  \frac{3}{(s+2)!} - \frac{3}{(s+1)!} + \frac{1}{(s-1)!} , \\
a_s 		&=& \frac{3}{(s+2)!} - \frac{3}{(s+1)!} + \frac{1}{s!},
\eea
then $p = s-2$ and $r = s+3$. 
\end{proposition}

According to Proposition \ref{prop:main}, when $p=2$, an $s$-stage method achieves the best energy accuracy $r=2s-1$ 
if the set of coefficients $\{a_k\}_{k=3}^s$ forms a solution to the system of equations: $b_k=0$ for $2\le k \le s-1$. 
For $3 \le s \le 5$, the coefficients are listed in Table \ref{tab:p2}, and the stability regions of these second-order methods and a fourth-order method are shown in Figure \ref{fig:p2}. The RK(4,4,5) method is the classical fourth-order RK method, known as RK4. 
Most methods have comparable or larger stability regions than RK4, except RK(4,2,7)-b and RK(5,2,9)-b. 

\begin{table}
\caption{Coefficients and the corresponding orders of accuracy when $p=2$.}
\label{tab:p2}
\begin{tabular}{c|ccc|ccccc} \hline \noalign{\smallskip}
method & $s$ & $p$ & $r$ & $a_1$ & $a_2$	& $a_3$ & $a_4$ & $a_5$ 					\\ 
\noalign{\smallskip} \hline \noalign{\smallskip}
RK(3,2,5)~~	& 3 & 2 & 5 & 1 & $\frac{1}{2}$ & $\frac{1}{8}$ & - & - 								\\ \noalign{\smallskip}
RK(4,2,7)-a 	& 4 & 2 & 7 & 1 & $\frac{1}{2}$ & $\frac{2 - \sqrt{2}}{4}$ & $\frac{3 - 2\sqrt{2}}{8}$ & - 	\\ \noalign{\smallskip} 
RK(4,2,7)-b 	& 4 & 2 & 7 & 1 & $\frac{1}{2}$ & $\frac{2 + \sqrt{2}}{4}$ & $\frac{3 + 2\sqrt{2}}{8}$ & - 	\\ \noalign{\smallskip} 
RK(5,2,9)-a 	& 5 & 2 & 9 & 1 & $\frac{1}{2}$ & $\frac{\sqrt{5}-1}{8}$ & $\frac{\sqrt{5}-2}{8}$ & $\frac{(\sqrt{5}-2)^2}{16(\sqrt{5}-1)}$  \\ \noalign{\smallskip}
RK(5,2,9)-b 	& 5 & 2 & 9 & 1 & $\frac{1}{2}$ & $\frac{1}{4}$ & $\frac{1}{8}$ & $\frac{1}{32}$ 			\\ 
\noalign{\smallskip} \hline
\end{tabular}
\end{table}

\begin{figure}
\begin{center}	
 \includegraphics[width=0.75\textwidth]{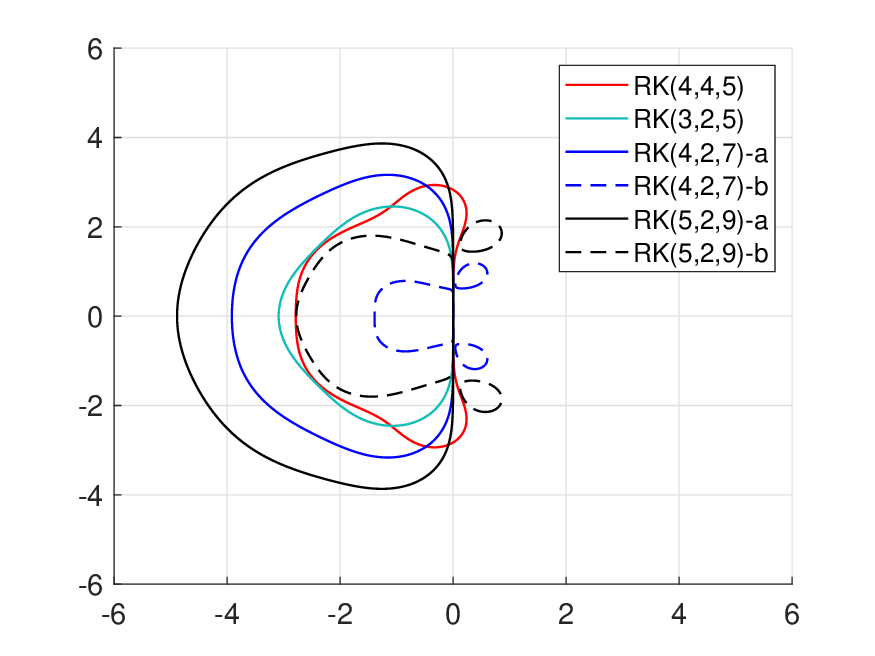}
 \caption{Stability regions of several second-order methods and the RK(4,4,5) method.}
 \label{fig:p2}
\end{center}
\end{figure}

When $p=2$, the leading coefficient $b_s = a_s^2 > 0$, and the method is not strongly stable \cite{sun2019strong}. 
To obtain strongly stable methods, we consider $p=4$, where the energy accuracy can reach an order of $r=2s - 3$, 
given that there exists a solution to the system of equations $b_k=0$ for $3\le k \le s-2$. 
For example, when $s=7$ we obtain 
\bea 
b_1 &=& a_1^2 - 2a_0a_2, \\
b_2 &=& a_2^2 - 2a_1a_3 + 2a_0a_4, \\
b_3 &=& a_3^2 - 2a_2a_4 + 2a_1a_5 - 2a_0a_6, \\
b_4 &=& a_4^2 - 2a_3a_5 + 2a_2a_6 - 2a_1a_7, \\
b_5 &=& a_5^2 - 2a_4a_6 + 2a_3a_7, \\
b_6 &=& a_6^2 - 2a_5a_7, \\
b_7 &=& a_7^2. 
\eea
If $p=4$, then $a_1=1$, $a_2=1/2$, $a_3=1/6$, $a_4=1/24$, and $b_1=b_2=0$. 
Thus, by setting $b_3=b_4=b_5=0$, we obtain a system of three equations with three unknowns $a_5$, $a_6$, and $a_7$:
\beq \label{eq:s7}
\left\{
\barr{ll}
2 a_5 - 2 a_6 - \frac{1}{72}						&= 0, 	\\
 -\frac{1}{3} a_5 + a_6 - 2a_7  + \frac{1}{576} 	&= 0, 	\\
a_5^2 - \frac{1}{12} a_6 + \frac{1}{3} a_7 		&= 0.	
\earr
\right.
\eeq
System (\ref{eq:s7}) has one positive solution: $a_5=\frac{\sqrt{10}-2}{144}$, $a_6=\frac{\sqrt{10}-3}{144}$, and $a_7=\frac{8\sqrt{10}-25}{3456}$. 
Furthermore, in this case, we can verify that the leading coefficient $b_6 = \frac{\sqrt{9604}-\sqrt{9610}}{248832} < 0$, so the method is strongly stable, as Theorem 2.7 in \cite{sun2019strong} has proven that a negative leading coefficient implies strong stability. 
Similarly, we can calculate the coefficients and verify the strong stability for $s=5$ and $s=6$.
Additionally, we can find the strong stability criterion using Lemma \ref{lemma1} and it is summarized in the following theorem.

\begin{theorem}\label{theorem1}
For an $s$-stage, fourth-order Runge--Kutta method applied to the system (\ref{eq:de}), 
if the coefficients $\{a_k\}_{k=3}^s$ form a solution to the system of equations $b_k=0$ for $3 \le k \le s-2$, and $b_{s-1}=a_{s-1}^2 - 2 a_s a_{s-2} < 0$, 
then the order of energy accuracy is $r = 2s-3$. Furthermore, the method is strongly stable if 
\beq
h||L|| \le \lambda = \sqrt{\frac{2 a_s a_{s-2} - a_{s-1}^2}{a_s^2}}. \label{eq:stability1}
\eeq
\end{theorem}

The energy equation (\ref{eq:energy}) yields
\beq
\mathcal{E}_{n+1} \le \mathcal{E}_n + \frac{1}{2} \sum_{k=m}^{s} b_k h^{2k} ||L^k||^2 ||u_n||^2 
				  = \left(1 + \sum_{k=m}^{s} b_k h^{2k} ||L^k||^2 \right) \mathcal{E}_n.
\eeq
Consequently, by induction, we have
\beq
\mathcal{E}_{n} \le \left(1 + \sum_{k=m}^{s} b_k h^{2k} ||L^k||^2 \right)^{n} \mathcal{E}_0,
\eeq
and the relative energy error satisfies
\begin{align}
\left|\frac{\mathcal{E}_{n}  - \mathcal{E}_0}{\mathcal{E}_0} \right|
	& \le n |b_m|\cdot||L^m||^2 \cdot h^{2m} + \mathcal{O}(h^{2m+1}), \\
	& = T |b_m|\cdot ||L^m||^2 \cdot h^{2m-1} + \mathcal{O}(h^{2m+1}),
\end{align}
where $T=nh$ is the final time. 
Furthermore, applying Equation (\ref{eq:bk}) and the triangle inequality, we obtain
\beq
|b_m| = \left|\sum_{i=\text{max}(0,2m-s)}^{\text{min}(2m,s)} (-1)^{m+i} a_i a_{2m-i} \right| 
	\le \sum_{i=\text{max}(0,2m-s)}^{\text{min}(2m,s)} |a_i a_{2m-i}|,
\eeq
which implies that $b_m$ is bounded.
Therefore, we establish the following theorem regarding the global energy error.

\begin{theorem}\label{theorem2}
For a Runge--Kutta method applied to the system (\ref{eq:de}), the relative energy error satisfies 
\beq
\left|\frac{\mathcal{E}_{n}  - \mathcal{E}_0}{\mathcal{E}_0} \right| \le T |b_m|\cdot ||L^m||^2 \cdot h^{2m-1} + \mathcal{O}(h^{2m+1}),
\eeq
where $m$ is the leading index in the energy equation (\ref{eq:energy}), $T=nh$ is the final time, 
and $b_m$ is  bounded by: 
\beq
|b_m| \le \sum_{i=\text{max}(0,2m-s)}^{\text{min}(2m,s)} |a_i a_{2m-i}|.
\eeq
\end{theorem}

Table \ref{tab:p4} lists the coefficients of four- to seven-stage methods of order four and the corresponding $\lambda$ values in the strong stability criterion (\ref{eq:stability1}). Only coefficients with indices larger than $p$ are listed, as $a_k=1/k!$ if $k\le p$. 
Additionally, the stability regions are illustrated in Figure \ref{fig:p4}. It demonstrates that as $s$ increases, the stability region expands.

\begin{table}
\begin{center}
\caption{Coefficients and stability criteria of fourth-order RK methods for $4 \le s\le 7$.}
\label{tab:p4}
\begin{tabular}{ccccccccc} 
\hline\noalign{\smallskip}
method & $s$ & $p$ & $r$ & $a_5$ & $a_6$ & $a_7$	& $b_{s-1}$ &$\lambda$	\\ 
\noalign{\smallskip} \hline\noalign{\smallskip}
RK(4,4,5)~~	& 4 & 4 & 5 &  -  & - & -								& $-\frac{1}{72}$ & $2\sqrt{2}$		\\ \noalign{\smallskip}
RK(5,4,7)~~	& 5 & 4 & 7 &  $\frac{1}{144}$  & - & -					& $-\frac{1}{1728}$ & $2\sqrt{3}$	\\ \noalign{\smallskip}
RK(6,4,9)~~	& 6 & 4 & 9 & $\frac{1}{128}$ & $\frac{1}{1152}$ & - 	& $-\frac{5}{442368}$ & $\sqrt{15}$		\\ \noalign{\smallskip}
RK(7,4,11)~~& 7 & 4 & 11 & $\frac{\sqrt{10}-2}{144}$ & $\frac{\sqrt{10}-3}{144}$ & $\frac{8\sqrt{10}-25}{3456}$	& $-\frac{\sqrt{9610}-\sqrt{9604}}{248832}$ & $\sqrt{\frac{48(31\sqrt{10}-98)}{5(253-80\sqrt{10})}}\approx 4.06$	 \\ 
\noalign{\smallskip} \hline
\end{tabular}
\end{center}
\end{table}

\begin{figure}
\begin{center}
 \includegraphics[width=0.75\textwidth]{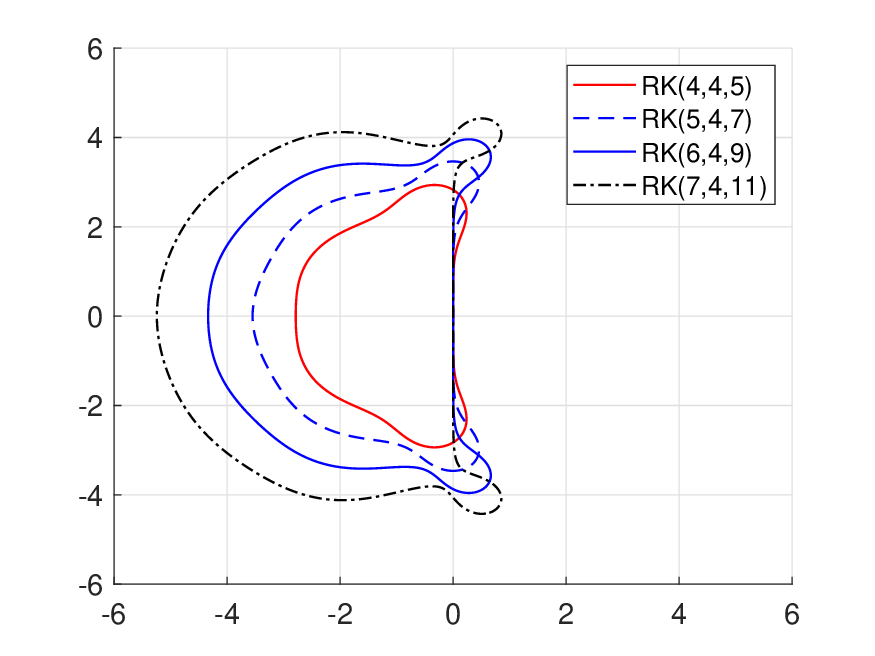}
 \caption{Stability regions of several fourth-order methods.}
 \label{fig:p4}
\end{center}
\end{figure}

Because the order of energy accuracy exceeds that of both the solution and the number of stages, we refer to it as the energy-superconvergence phenomenon, and we term the corresponding RK method the energy-superconvergent RK (ESC-RK) method. 

To construct an $s$-stage RK method using the coefficients $\{a_k\}$, we can use the following algorithm: 
\beq
\begin{array}{ll}
k_0 		  &= 0, \\
k_j 		  &= c_j hL(u_n + k_{j-1}), ~~ j=1, 2, ..., s, \\
u_{n+1}  &= u_n + k_s,
\end{array} 
\eeq
where 
\beq
c_j = \frac{a_{s-j+1}}{a_{s-j}},~~ j=1, 2, ..., s.
\eeq

Another approach to construct an $s$-stage RK method using the coefficients $\{a_k\}$ is through the generalized iterated Crank-Nicolson procedure \cite{liu2024iterated2}. The algorithm proceeds as follows: 
\begin{align}
k_1 &= L(u_n),                													\label{eq_icn_s1} \\
k_{j+1} &= L( u_n + c_j h ( (1-\theta_j) k_1 + \theta_j k_j ) ), \quad j=1,2,...,s-1, 	\label{eq_icn_s2}\\
u_{n+1} &= u_n + c_s h ( (1-\theta_s) k_1 + \theta_s k_s),						\label{eq_icn_s3}
\end{align}
where the coefficients $c_j$ are defined by
\beq
c_j = 2^j \prod_{i=1}^{j} {\theta_i},\quad j=1,2,\ldots,s.
\eeq 
With $\theta_1=0.5$ and $c_1=2\theta_1=1$, the remaining coefficients $\{\theta_j\}$ and $\{c_j\}$ are determined by:
\begin{align}
c_j &= \frac{2a_{s-j+2}}{a_{s-j+1}}, \label{eq:c_j}\\
\theta_j &= \frac{c_j}{2c_{j-1}}, 	\label{eq:theta_j}
\end{align}
for $j=2,3,\ldots,s$.

From equations (\ref{eq:c_j}) and (\ref{eq:theta_j}), we derive the relation
\beq
a_s a_{s-2} = \frac{c_2 a_{s-1}}{2} \frac{2a_{s-1}}{c_3} = \frac{a_{s-1}^2}{2\theta_3}.
\eeq
When the leading index $m=s-1$, we obtain 
\beq
b_{s-1} 	= a_{s-1}^2 - 2a_s a_{s-2}  
		= a_{s-1}^2 \left(1 - \frac{1}{\theta_3} \right). \label{eq:bmtheta}
\eeq
This implies that $b_{s-1}$ is bounded above by $a_{s-1}^2 \left|1 - \frac{1}{\theta_3} \right|$.
Furthermore, when $0 < \theta_3 < 1$, we have $b_{s-1} < 0$, which satisfies the strong stability condition required by Theorem \ref{theorem1}.

	Furthermore, we can apply our results to analyze the stability and energy order of the original ICN method \cite{Teukolsky00}. 
	The original ICN method \cite{Teukolsky00} is a special case of the general ICN method (\ref{eq_icn_s1})-(\ref{eq_icn_s3}) when $\theta_j$ is fixed at 0.5: 
	\begin{align}
		k_1 &= L(u_n),                											\label{eq_icn_eq1} \\
		k_{j+1} &= L\left( u_n + \frac{h}{2} ( k_1 + k_j ) \right), \quad j=1,2,...,s-1, 	\label{eq_icn_eq2}\\
		u_{n+1} &= u_n + \frac{h}{2} ( k_1 + k_s).								\label{eq_icn_eq3}
	\end{align}
	We have $a_j = 2^{-j+1}$ for $j\ge 1$.  
	Thus, $a_1=1$, $a_2=1/2$, and $a_3=1/4 \ne 1/6$, so the original ICN method is only second order accuracy.
	
	Using Equation (\ref{eq:bk}) (the definition of $b_k$), we obtain
	\begin{equation}
		a_{k-i}a_{k+i} = \left\{
		\begin{array}{ll}
			2^{-2k+2}, & i<k, \\
			2^{-2k+1}, & i=k,
		\end{array} \right.
	\end{equation}
	and
	\begin{equation}
		b_k = \left\{
		\begin{array}{ll}
			0, 											& k \le \lfloor s/2 \rfloor, \\
			2^{-2k+2}[1 - 2 ~\text{mod}(s-k,2)] \ne 0, 	& k > 	\lfloor s/2 \rfloor,
		\end{array} \right.
	\end{equation}
	where $\text{mod}$ represents the modulo operation.
	Therefore, the leading index is $m=\lfloor s/2 \rfloor + 1$ and 
	the energy order of $s$-stage ICN method is summarized in the following theorem.
	\begin{theorem}\label{theorem_ICN_energyorder}
		For the system (\ref{eq:de}), the energy order of an $s$-stage iterated Crank-Nicolson method (\ref{eq_icn_eq1})-(\ref{eq_icn_eq3}) is given by
		$r=2 \lfloor s/2 \rfloor + 1$. Equivalently,
		\begin{equation}
			r = \left\{
			\begin{array}{ll}
				s, 	 & s~ \text{is odd}, \\
				s+1, & s~ \text{is even}.
			\end{array} \right.
		\end{equation}
	\end{theorem}
	This result is supported by numerical experiments presented in \cite{liu2024iterated}. 
	
	Furthermore,
	$s-m$ is even when $s \equiv 1~(\text{mod} ~4)~or~s \equiv 2~ (\text{mod} ~4)$ and 
	$s-m$ is odd  when $s \equiv 0~(\text{mod} ~4)~or~s \equiv 3~ (\text{mod} ~4)$. 
	Thus,
	\begin{equation}
		b_m = \left\{
		\begin{array}{ll}
			2^{-2\lfloor s/2 \rfloor}, & s \equiv 1~(\text{mod} ~4)~or~s \equiv 2~ (\text{mod} ~4), \\
			-2^{-2\lfloor s/2 \rfloor}, & s \equiv 0~(\text{mod} ~4)~or~s \equiv 3~ (\text{mod} ~4).
		\end{array} \right.
	\end{equation}
	When $s \equiv 0~(\text{mod}~4)$ or $s \equiv 3~ (\text{mod}~4)$, we have $b_m<0$ and the method is strongly stable.
	This property was previously noted in \cite{Teukolsky00}.  
	

\section{ Energy-superconvergent RK for nonlinear problems}
\label{sec:nonlinear}

In this section, we discuss the extension of energy-superconvergent RK methods to the nonlinear ODE $U_t = F(U)$. 
The $s$-stage RK method is written as
\begin{align}
	k_j &= F\left(U_n + h \sum_{i=1}^{j-1} \alpha_{ji} k_i \right), \\
	U_{n+1} &= U_n + h \sum_{i=1}^{s} {\beta_i k_i}.
\end{align}
We begin with three-stage methods with fifth-order energy accuracy, followed by four- and five-stage methods achieving seventh-order energy accuracy.

\subsection{Three-stage energy-superconvergent methods}
We consider the following autonomous nonlinear system with amplitude-dependent frequency
\beq
\vtwo{u}{v}_t = \mtwo{0}{A(u,v)}{-A(u,v)}{0} \vtwo{u}{v}, \label{eq:vode}
\eeq
where $A(u,v)=A(|U|^2) = A(u^2+v^2)$. 
The energy of this system is given by
\beq
\mathcal{E} = \frac{1}{2} (u^2 + v^2).
\eeq
Since $A$ is amplitude-dependent, it follows that $u\frac{\partial A}{\partial v} = v\frac{\partial A}{\partial u}$ (or $uA_v=vA_u$). 
When $s=3$, using a Taylor expansion and applying the second-order conditions for the RK algorithm, 
we obtain 
\begin{equation} \label{eq:orderRK325}
\mathcal{E}_{n+1} = \mathcal{E}_n \left[ 1 + \frac{h^4}{4} \delta_1 A^4 + h^4 \delta_2 A^3(uA_u+vA_v)\right] + O(h^6),
\end{equation}
where
\begin{align}
	\delta_1 &= 1 - \beta_3 \alpha_{32}c_2,\\
	\delta_2 &= (\beta_2 c_2^2 + \beta_3 c_3^2) - (\beta_2 c_2^3 + \beta_3 c_3^3) + \beta_3 \alpha_{32} c_2 (2c_3 - c_2 - 2).
\end{align}
Here,
$c_i = \sum_{j=1}^{i-1}{\alpha_{ij}}$. 
When $\delta_1=\delta_2=0$, the fourth-order ($h^4$) terms  in Equation (\ref{eq:orderRK325}) vanish and the energy accuracy is of fifth order ($r=5$). 
Consequently, we obtain the following order conditions for a three-stage RK method to achieve a solution order of $p=2$ and an energy order of $r=5$:
\begin{align}
	\beta_1 + \beta_2 + \beta_3 &= 1, 				\label{eq:order1} \\
	\beta_2 c_2 + \beta_3 c_3 &= \frac{1}{2},		\label{eq:order2} \\
	1 - 8\beta_3\alpha_{32}c_2 &= 0,				\label{eq:energyorder1} \\
	(\beta_2 c_2^2 + \beta_3 c_3^2) - (\beta_2 c_2^3 + \beta_3 c_3^3) + \beta_3 \alpha_{32} c_2 (2 c_3 - c_2 - 2) &= 0. \label{eq:energyorder2}
\end{align}
Equation (\ref{eq:energyorder1}) is the fifth-order energy condition for the linear case: $\beta_3\alpha_{32}c_2 = a_3 = 1/8$.
For a linear problem, $A$ is a constant matrix, so $A_u=A_v=0$ and the second $h^4$ term in Equation (\ref{eq:orderRK325}) vanishes. 
Therefore, Equation (\ref{eq:energyorder2}) constitutes the additional order condition arising from the nonlinearity. 

The solutions to system (\ref{eq:order1})-(\ref{eq:energyorder2})  can be expressed as 
\begin{align}
	\beta_3 	&= \frac{4(c_2^3-c_2^2) + c_2(2+c_2-2c_3)}{8c_2c_3(c_3-c_2)(1-c_2-c_3)}, \\
	\beta_2 	&= \frac{4(c_3^2-c_3^3) - c_3(2+c_2-2c_3)}{8c_2c_3(c_3-c_2)(1-c_2-c_3)}, \\
	\beta_1 	& = 1 - \beta_2 - \beta_3,\\
	\alpha_{32} 	&= \frac{1}{8\beta_3c_2},\\
	\alpha_{31} 	&= c_3 - \alpha_{32},\\
	\alpha_{21} 	&= c_2.
\end{align}
When $c_2=1/2$ and $c_3=1$, we obtain the following set of coefficients:
\beq
\alpha_{21} = \frac{1}{2},~\alpha_{31}=0,~\alpha_{32}=1,
~\beta_1=\frac{1}{4},~\beta_2=\frac{1}{2},~\beta_3=\frac{1}{4},
~c_2=\frac{1}{2},~c_3 = 1,
\eeq
and the corresponding RK method is: 
\begin{align}
	k_1 &= F(U_n), 										\label{eq:rk325_1}\\
	k_2 &= F(U_n + \frac{h}{2} k_1 ), 					\label{eq:rk325_2}\\
	k_3 &= F(U_n + h k_2 ), 							\label{eq:rk325_3}\\
	U_{n+1} &= U_n + \frac{h}{4}( k_1 + 2 k_2 + k_3).	\label{eq:rk325_4}
\end{align}

This method is referred to as RK325, reflecting its three stages ($s=3$), second-order solution accuracy ($p=2$), 
and fifth-order energy accuracy ($r=5$). 

For the nonlinear cases, we use the notation RK$spr$ instead of RK($s,p,r$). RK($s,p,r$) refers to a family of RK methods developed in the previous section for linear systems, whereas RK$spr$ denotes a specific energy-superconvergent explicit RK method designed for autonomous nonlinear systems with amplitude-dependent frequencies.

\subsection{Four-stage and five-stage energy-superconvergent methods}

For four- and five-stage methods, we consider a special case of the autonomous nonlinear system (\ref{eq:vode}) 
where $A$ is linear with respect to $u^2+v^2$:
\begin{equation}
	A(u^2+v^2) = a + b (u^2+v^2),
\end{equation}
with constants $a$ and $b$. 

To achieve seventh-order energy accuracy, the Taylor expansion and order conditions lead to a large system of nonlinear equations. 
Specifically, when $s=4$, $p=2$, and $r=7$, the system is provided in Appendix A; when $s=5$, $p=4$, and $r=7$, the system is given in Appendix B. 
Solving these nonlinear equations numerically, we obtain three RK methods with seventh-order energy accuracy. 
Two are four-stage, second-order methods, referred to as RK427a and RK427b, while the third one is a five-stage, fourth-order method designated as RK547. 
The coefficients are provided in Table \ref{tab:r7}. 
The remaining coefficients ($\alpha_{j1}$ and $\beta_1$) can be determined using the following relations:
\begin{align} 
	\alpha_{21} &= c_2,												\label{eq:a21}	\\
	\alpha_{31} &= c_3 - \alpha_{32},								\label{eq:a31}	\\
	\alpha_{41} &= c_4 - \alpha_{43} - \alpha_{42},					\label{eq:a41}	\\
	\alpha_{51} &= c_5 - \alpha_{54} - \alpha_{53} - \alpha_{52},	\label{eq:a51}	\\
	\beta_1 	&= 1 - \beta_2 - \beta_3 - \beta_4 - \beta_5. 		\label{eq:b1}
\end{align}

\begin{table}							
	\begin{center}
	\caption{Coefficients of three RK methods of energy order seven (RK427a, RK427b, and RK547). 
		Coefficients that are not in this table can be calculated using system (\ref{eq:a21})-(\ref{eq:b1}). }							
		\label	{tab:r7}						
		\begin{tabular}{llll}					
			\hline \noalign{\smallskip}							
			Method	&	RK427a	&	RK427b	&	RK547	\\
			\hline \noalign{\smallskip}							
			$c_2$			&	0.5					&	0.25				&	~0.20892886718970132831	\\
			$c_3$			&	1.126707539929660	&	0.665773693052985	&	~0.94900422371489578932	\\
			$c_4$			&	0.25				&	1					&	-0.07278204742298131913	\\
			$c_5$			&						&						&	~0.68134086764041323914	\\
			$\alpha_{32}$	&	1.707869936784730	&	0.684915394057140	&	~0.94900422371489578932	\\
			$\alpha_{42}$	&	0					&	0					&	~0.28579013534165120802	\\
			$\alpha_{43}$	&	0.122516522451472	&	0.738611266763089	&	-0.35857218276463254103	\\
			$\alpha_{52}$	&						&						&	~0.72441810631776648588	\\
			$\alpha_{53}$	&						&						&	~0.18811713344639199863	\\
			$\alpha_{54}$	&						&						&	-0.23119437212374524537	\\
			$\beta_2$		&	0.585723950941299	&	0.340967677611324	&	~0.42481264428380438591	\\
			$\beta_3$		&	0.138358669923910	&	0.368265583183962	&	~0.13163010989793449967	\\
			$\beta_4$		&	0.204993071645761	&	0.169576543256471	&	~0.02106663674573944212	\\
			$\beta_5$		&						&						&	~0.42249060907252167230	\\
			\hline \noalign{\smallskip}							
		\end{tabular}							
	\end{center}
\end{table}


\section{Linear examples}
\label{sec:examples}

In this section, various linear ODE systems are simulated using the proposed methods, including second-order ODEs for harmonic oscillators, one-dimensional integro-differential equation for peridynamics, and the ODE systems derived from the semi-discretized one-dimensional Maxwell's equations of electrodynamics. 

The convergence of the solution is measured using standard $L_1$, $L_2$, and $L_\infty$ error norms, denoted by $\epsilon_1$, $\epsilon_2$, and $\epsilon_\infty$, respectively. 
To test the energy accuracy, we calculate the convergence rate of the relative energy deviation:
\beq
\epsilon_E = \frac{\mathcal{E}_T - \mathcal{E}_0}{\mathcal{E}_0},
\eeq
where $\mathcal{E}_0$ and $\mathcal{E}_T$ represent the energy at the initial and the final times, respectively. 

\subsection{Second-order differential equation for harmonic oscillator}
\label{sec:ODE}

In the first example, we consider the second-order differential equation for a harmonic oscillator: 
\beq 
\begin{array}{ll}
& x''(t) +a^2 x(t) = 0, ~ 0 \le t \le T,~ \label{eq:ex1} \\
& x(0) = x_0, ~ x'(0) = v_0.
\end{array}
\eeq 
This equation can be written in matrix form as: 
\beq
\frac{d}{dt} u = \mtwo{0}{1}{-a^2}{0} u, \label{eq:ex1matrix}
\eeq
where $u=\vtwo{x}{v} $ and $v(t)=x'(t)$. 
The total energy of this system is given by:
\beq
\mathcal{E}=\frac{1}{2}\innerproduct{u}{u}_H = \frac{1}{2} (a^2 x^2 + v^2),
\eeq
where $H = \mtwo{a^2}{0}{0}{1}$.

In our simulations, we set $a=1$, $T=80$, $x_0=1$, and $v_0=0$. The exact solution is $x(t) = x_0 \cos(a t)$. 
We use a uniform time step $\Delta t = T/N_t$ for all methods, where $N_t$ represents the number of grid points.
   
Table \ref{tab:test1p2} presents the results obtained from several second-order RK methods, with coefficients provided in Table \ref{tab:p2}, 
along with the St{\"o}rmer--Verlet (SV) method \cite{hairer2003geometric}. 
We adopt RK(4,2,7)-a and RK(5,2,9)-a due to their larger stability regions compared to their corresponding b-versions. 
These results illustrate the convergence rates of these methods, including their orders of energy accuracy. 
Notably, when $N_t=1600$, the relative energy deviation of RK(5,2,9) achieves machine precision ($\sim 10^{-15}$). 
Regarding the solution accuracy, the RK(3,2,5) has similar accuracy to the SV method, and as $s$ increases, the errors decrease. 
For this set of second-order RK methods, the relative energy deviations are positive, indicating that these second-order methods are not strongly stable. 

The results of fourth-order methods are presented in Table \ref{tab:test1p4}. 
Similar to the second-order methods, the errors decrease as $s$ increases. 
However, unlike the second-order methods, the relative energy deviations are all negative, indicating numerical dissipation. 
Table \ref{tab:test1cpu} compares the computational performance of four fourth-order RK methods used to integrate the harmonic oscillator to a target energy accuracy of $|\epsilon_E| < 1 \times 10^{-13}$.
As the number of stages in the RK method increases, the computational efficiency regarding energy accuracy improves.
The seven-stage method RK(7,4,11) permits a step size that is 72 times larger than the four-stage method RK(4,4,5), resulting in a simulation that is 40 times faster.



\begin{table}																			
\caption{Solution and energy convergence rates of harmonic oscillator simulations using several second-order methods.} 				
\label{tab:test1p2}																		
\begin{tabular}{cccccccccc}	\hline\noalign{\smallskip}
Method	& $N_t$	&$\epsilon_1$ 	&	Ratio	&	$\epsilon_2$&	Ratio	&$\epsilon_\infty$&	Ratio	&$\epsilon_E$	&	Ratio	\\ 
		\hline\noalign{\smallskip}					
SV		&	100	&	6.62E-01	&			&	8.30E-01	&			&	1.77E+00	&			&	-5.43E-02	&			\\						
		&	200	&	1.74E-01	&	1.93	&	2.22E-01	&	1.90	&	5.30E-01	&	1.74	&	-3.30E-02	&	0.72	\\						
		&	400	&	4.28E-02	&	2.02	&	5.48E-02	&	2.02	&	1.34E-01	&	1.99	&	-9.99E-03	&	1.72	\\						
		&	800	&	1.06E-02	&	2.01	&	1.36E-02	&	2.01	&	3.32E-02	&	2.01	&	-2.49E-03	&	2.01	\\						
		&	1600&	2.65E-03	&	2.00	&	3.40E-03	&	2.00	&	8.29E-03	&	2.00	&	-6.18E-04	&	2.01	\\	 \hline\noalign{\smallskip}						
RK(3,2,5)&	100	&	7.54E-01	&			&	9.49E-01	&			&	2.02E+00	&			&	5.05E-01	&			\\						
		&	200	&	1.79E-01	&	2.08	&	2.29E-01	&	2.05	&	5.46E-01	&	1.89	&	1.29E-02	&	5.29	\\						
		&	400	&	4.31E-02	&	2.05	&	5.52E-02	&	2.05	&	1.35E-01	&	2.02	&	4.00E-04	&	5.01	\\						
		&	800	&	1.06E-02	&	2.02	&	1.37E-02	&	2.02	&	3.33E-02	&	2.02	&	1.25E-05	&	5.00	\\						
		&	1600&	2.65E-03	&	2.00	&	3.40E-03	&	2.00	&	8.29E-03	&	2.00	&	3.91E-07	&	5.00	\\	 
		\hline\noalign{\smallskip}						
RK(4,2,7)&	100	&	3.49E-01	&			&	4.45E-01		&		&	9.70E-01	&			&	7.75E-03	&			\\						
		&	200	&	8.41E-02	&	2.05	&	1.08E-01	&	2.04	&	2.62E-01	&	1.89	&	6.03E-05	&	7.01	\\						
		&	400	&	2.07E-02	&	2.02	&	2.66E-02	&	2.02	&	6.47E-02	&	2.02	&	4.71E-07	&	7.00	\\						
		&	800	&	5.16E-03	&	2.01	&	6.61E-03	&	2.01	&	1.61E-02	&	2.01	&	3.68E-09	&	7.00	\\						
		&	1600&	1.29E-03	&	2.00	&	1.65E-03	&	2.00	&	4.02E-03	&	2.00	&	2.87E-11	&	7.00	\\	 \hline\noalign{\smallskip}					
RK(5,2,9)&	100	&	2.05E-01	&			&	2.64E-01	&			&	6.17E-01	&			&	8.53E-05	&			\\						
		&	200	&	5.03E-02	&	2.03	&	6.44E-02	&	2.03	&	1.57E-01	&	1.98	&	1.67E-07	&	9.00	\\						
		&	400	&	1.24E-02	&	2.01	&	1.59E-02	&	2.01	&	3.88E-02	&	2.01	&	3.25E-10	&	9.00	\\						
		&	800	&	3.10E-03	&	2.01	&	3.97E-03	&	2.01	&	9.68E-03	&	2.00	&	6.31E-13	&	9.01	\\						
		&	1600&	7.74E-04	&	2.00	&	9.92E-04	&	2.00	&	2.42E-03	&	2.00	&$\sim10^{-15}$	&	-		\\	\hline					
\end{tabular}																							
\end{table}

\begin{table}																								
\caption{Solution and energy convergence rates of harmonic oscillator simulations using several fourth-order RK methods.} 	
\label{tab:test1p4}																									
\begin{tabular}{cccccccccc}																								
\hline\noalign{\smallskip}
Method	&$N_t$	&$\epsilon_1$&	Ratio	&$\epsilon_2$	&	Ratio	&$\epsilon_\infty$&	Ratio	&$\epsilon_E$	&	Ratio	\\	\hline\noalign{\smallskip}						
RK(4,4,5)&100&	8.24E-02	&			&	1.04E-01	&			&	2.40E-01	&			&	-2.85E-01	&			\\			
	&	200	&	5.43E-03	&	3.92	&	6.94E-03	&	3.91	&	1.63E-02	&	3.88	&	-1.11E-02	&	4.68	\\			
	&	400	&	3.39E-04	&	4.00	&	4.35E-04	&	4.00	&	1.03E-03	&	3.99	&	-3.54E-04	&	4.97	\\			
	&	800	&	2.12E-05	&	4.00	&	2.72E-05	&	4.00	&	6.54E-05	&	3.97	&	-1.11E-05	&	4.99	\\			
	&	1600&	1.32E-06	&	4.00	&	1.70E-06	&	4.00	&	4.12E-06	&	3.99	&	-3.47E-07	&	5.00	\\	\hline\noalign{\smallskip}		
RK(5,4,7)&100 &	1.78E-02	&			&	2.26E-02	&			&	5.34E-02	&			&	-9.15E-03	&			\\			
	&	200	&	9.62E-04	&	4.21	&	1.23E-03	&	4.20	&	2.98E-03	&	4.16	&	-7.48E-05	&	6.93	\\			
	&	400	&	5.75E-05	&	4.06	&	7.37E-05	&	4.06	&	1.79E-04	&	4.06	&	-5.91E-07	&	6.99	\\			
	&	800	&	3.55E-06	&	4.02	&	4.55E-06	&	4.02	&	1.11E-05	&	4.01	&	-4.63E-09	&	7.00	\\			
	&	1600&	2.21E-07	&	4.01	&	2.83E-07	&	4.01	&	6.91E-07	&	4.00	&	-3.62E-11	&	7.00	\\	\hline\noalign{\smallskip}		
RK(6,4,9)&100&	6.01E-03	&			&	7.66E-03	&			&	1.84E-02	&			&	-1.16E-04	&			\\			
	&	200	&	3.49E-04	&	4.11	&	4.46E-04	&	4.10	&	1.08E-03	&	4.09	&	-2.35E-07	&	8.95	\\			
	&	400	&	2.14E-05	&	4.03	&	2.74E-05	&	4.03	&	6.66E-05	&	4.02	&	-4.62E-10	&	8.99	\\			
	&	800	&	1.33E-06	&	4.01	&	1.70E-06	&	4.01	&	4.15E-06	&	4.01	&	-9.03E-13	&	9.00	\\			
	&	1600&	8.29E-08	&	4.00	&	1.06E-07	&	4.00	&	2.59E-07	&	4.00	&	$\sim10^{-15}$	&	-	\\	\hline\noalign{\smallskip}		
RK(7,4,11)&	100	&	2.92E-03&			&	3.72E-03	&			&	8.94E-03	&			&	-8.13E-07	&			\\			
	&	200	&	1.74E-04	&	4.07	&	2.23E-04	&	4.06	&	5.39E-04	&	4.05	&	-4.09E-10	&	10.96	\\			
	&	400	&	1.07E-05	&	4.02	&	1.38E-05	&	4.02	&	3.34E-05	&	4.01	&	-2.03E-13	&	10.97	\\			
	&	800	&	6.68E-07	&	4.01	&	8.56E-07	&	4.01	&	2.08E-06	&	4.00	&	$\sim10^{-15}$	&	-	\\			
	&	1600&	4.17E-08	&	4.00	&	5.34E-08	&	4.00	&	1.30E-07	&	4.00	&	$\sim10^{-15}$	&	-	\\	\hline		
\end{tabular}																									
\end{table}				
				
\begin{table}										
	\begin{center}										
		\caption{CPU runtime comparison of four fourth-order RK methods achieving a target energy accuracy of $|\epsilon_E| < 1 \times10^{-13}$ for the harmonic oscillator integrated up to $T=80$.}										
		\label{tab:test1cpu}										
		\begin{tabular}{lrccc}										
			\hline \noalign{\smallskip}										
			method 		&	 $N_t$	&$\Delta t$	&$\epsilon_E$	&	CPU time (s) \\ \hline\noalign{\smallskip}
			RK(4,4,5)	&	32000	&	0.003	&	-9.06E-14	&	1.700		\\ \noalign{\smallskip}
			RK(5,4,7)	&	3790	&	0.021	&	-8.55E-14	&	0.244		\\ \noalign{\smallskip}
			RK(6,4,9)	&	1040	&	0.077	&	-8.13E-14	&	0.083		\\ \noalign{\smallskip}
			RK(7,4,11) 	&	440		&	0.182	&	-6.78E-14	&	0.042		\\ \hline
			
		\end{tabular}										
	\end{center}									
\end{table}



\subsection{Linear peridynamic integro-differential equation}
\label{sec:PD}

The theory of peridynamics \cite{Silling2000} can be employed to model the nonlocal wave interaction of a homogeneous and infinitely long bar. 
The governing equation is a one-dimensional nonlocal integro-differential equation, which in the linear case, can be written as
\beq
    \rho(x) u_{tt}(x,t) = \int_{-\infty}^\infty C(\Tilde{x}-x) (u(\Tilde{x},t) - u(x,t)) d\Tilde{x} + b(x,t),
    \label{eq:PD}
\eeq
where $u$, $\rho$, $b$, and $C$ represent the displacement field, the density, the external force, and the micromodulus function, respectively. 
In this test, we utilize a normal distribution for the micromodulus function 
\beq
C(x) = \left\{
    \begin{array}{ll}
	\frac{4}{\sqrt{\pi}} e^{-x^2}, 	& |x| < \delta, \\
      0, 								& \text{otherwise,}
    \end{array} \right. 
\eeq
where $\delta > 0$ represents the horizon.
To discretize Equation (\ref{eq:PD}), we partition the domain $[a,b]$ into $N_x$ cells and define $x_i$ at the cell centers: $x_i=a + (j-1/2)\Delta x$, $j=1,2,...,N_x$, where $\Delta x= (b-a)/N_x$. 
The integral in Equation (\ref{eq:PD}) can be approximated using the midpoint quadrature method: 
\beq
\int_{-\infty}^\infty C(\Tilde{x}-x_i) (u(\Tilde{x},t) - u(x_i,t)) d\Tilde{x} \approx \sum_{j=1}^{N_x}  C(x_j-x_i) (u_j - u_i) \Delta x.
\eeq
Let $U = (u_1,u_2,...,u_{N_x})^\top$, and assuming $\rho$ is constant, we can write the semi-discretized equation in matrix form: 
\beq
U_{tt} = -A U+ B, \label{eq:PDmatrix}
\eeq
where 
\beq 
B=\frac{1}{\rho}\left(b(x_1,t),...,b(x_{N_x},t)\right)^\top,
\eeq
and the entries of the matrix $A$ is defined by 
\beq
A_{ij} = \left\{
\begin{array}{ll}
-\frac{\Delta x}{\rho}C(x_j-x_i), 				&~ i \ne j,\\
\frac{\Delta x}{\rho}\sum_{k\ne i} C(x_k-x_i),	&~ i = j.
\end{array}\right.
\eeq 
System (\ref{eq:PDmatrix}) can be written as a system of first-order equations:
\bea
	\pdt U &=& V, \\
	\pdt V &=& -A U + B,
\eea
and in its matrix form, we have: 
\beq \label{eq:PDvec}
\pdt \vtwo{U}{V} = \mtwo{0}{I}{-A}{0} \vtwo{U}{V} + \vtwo{0}{B}.
\eeq 
The total energy (Hamiltonian) of the system is:
\beq
\mathcal{E} = \frac{1}{2} V^\top V + \frac{1}{2} U^\top A U - U^\top B. 
\eeq				

In our simulations, we set $b=0$, $\rho=1$, and $\delta = 5$. 
The computational domain is $x\in[-20,20]$ and $t\in[0,T]$. 
We let $\Delta t=\Delta x$. 
The initial condition is $u(x,0)=exp(-x^2)$ and $u_t(x,0)=0$. 
Similar to previous studies \cite{coclite2020numerical,liu2024iterated}, a periodic boundary condition is implemented to approximate the original problem of an infinitely long bar. 
The simulations run until $T=5$ so that they stop before the solution reaches the boundaries. 
The exact solution is given by \cite{weckner2005effect}:
\begin{equation}
    u_e(x,t)=\frac{2}{\sqrt{\pi}} \int_0^\infty {e^{-\xi^2} \cos(2x\xi) \cos\left(2t\sqrt{1-e^{-\xi^2}}\right)} d\xi.
\end{equation}

The solutions at two time snapshots using the seven-stage method RK(7,4,11) are shown in Figure \ref{fig:sol_PD}. 
Results on convergence rates are presented in Table \ref{tab:PD}. 
All methods provide fourth-order accurate solutions, but the accuracy improves with larger $s$ due to better energy accuracy.
Figure \ref{fig:long_sol_err_PD} shows the time history of the solution $L_2$ error norms for the SV method and four 4th-order RK methods with varying orders of energy accuracy. The simulations were run for an extended duration, up to $t=10$. 
The fourth-order RK method with a higher energy order consistently yields smaller solution errors than both the SV method and the RK methods with lower energy orders.


	Table \ref{tab:PDcpu} compares the computational performance of four fourth-order RK methods 
	required to achieve a target energy accuracy of $|\epsilon_E| < 5 \times 10^{-10}$.
	As the number of stages in the RK method increases, the computational efficiency regarding energy accuracy improves dramatically. 
	Specifically, the seven-stage method RK(7,4,11) permits a 24 times reduction in the number of spatial nodes ($N_x$) 
	and a 46 times reduction in time steps ($N_t$) compared to the four-stage RK(4,4,5) method. 
	This reduces runtime by a factor of more than 2,400.

\begin{figure}
\begin{center}
\subfigure[]{ \includegraphics[width=0.45\textwidth]{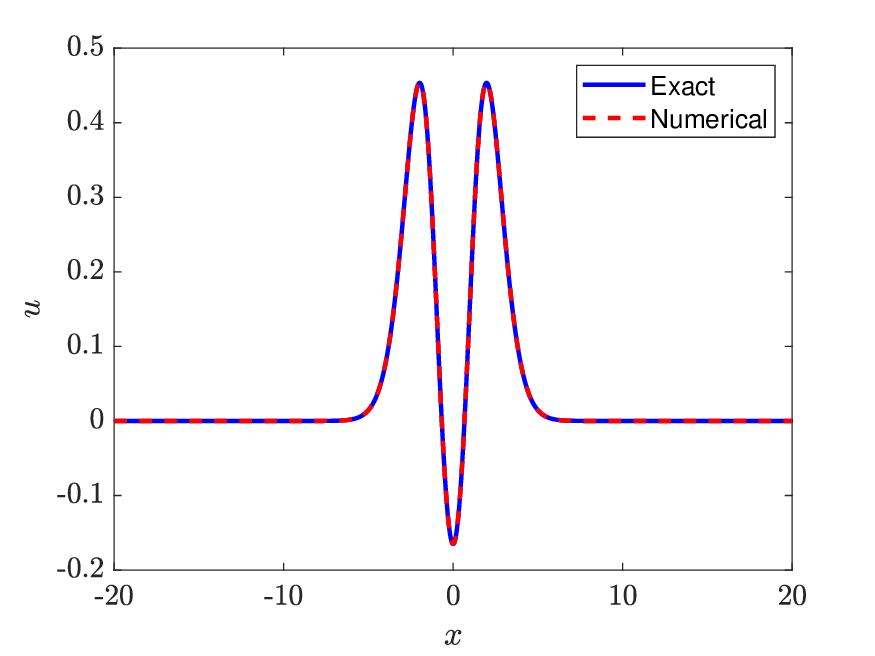} }
\subfigure[]{ \includegraphics[width=0.45\textwidth]{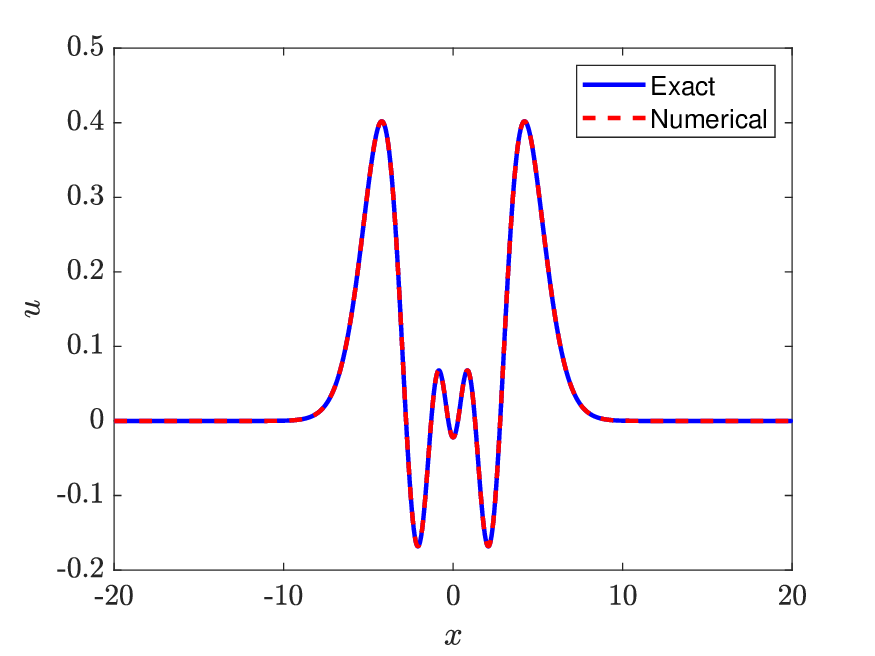} }
 \caption{Solutions of the linear peridynamic equation at (a) $t=2.5$ and (b) $t=5$ using RK(7,4,11).}
 \label{fig:sol_PD}
 \end{center}
\end{figure}

\begin{figure}
\begin{center}
 \includegraphics[width=0.6\textwidth]{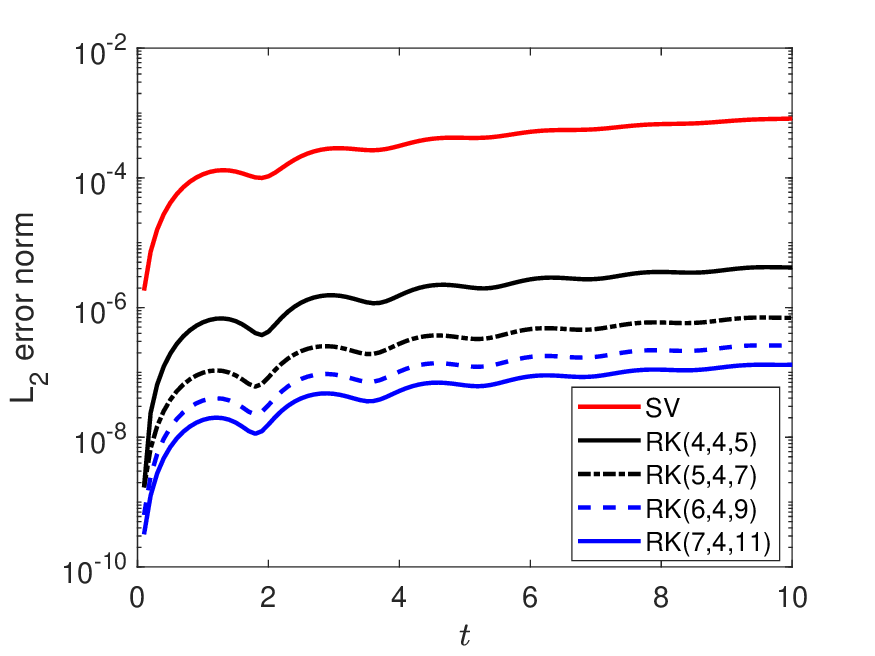} 
 \caption{Time history of the solution $L_2$ error norms of the linear peridynamic equation until $t=10$ using the SV method and four fourth-order RK methods: RK(4,4,5), RK(5,4,7), RK(6,4,9), and RK(7,4,11). The mesh size is $N_x=400$.}
 \label{fig:long_sol_err_PD}
 \end{center}
\end{figure}

\begin{table}																									
\caption{Solution and energy convergence rates of peridynamic integro-differential equations using fourth-order methods.}
\label{tab:PD}											
\begin{tabular}																						
	{ c		c		c		c		c		c		c		c		c		c	}	\hline\noalign{\smallskip}		
Method	&	N	&	$\epsilon_1$	&	Ratio	&	$\epsilon_2$	&	Ratio	&	$\epsilon_\infty$	&	Ratio	&	$\epsilon_E$	&	Ratio	\\	\hline\noalign{\smallskip}		
RK(4,4,5)&	100	&	2.62E-04	&	-		&	5.70E-04	&	-		&	2.05E-03	&	-		&	-5.86E-03	&	-		\\			
		&	200	&	1.58E-05	&	4.05	&	3.50E-05	&	4.03	&	1.26E-04	&	4.03	&	-1.85E-04	&	4.99	\\			
		&	400	&	9.63E-07	&	4.04	&	2.10E-06	&	4.06	&	7.52E-06	&	4.06	&	-5.84E-06	&	4.98	\\			
		&	800	&	5.93E-08	&	4.02	&	1.28E-07	&	4.03	&	4.62E-07	&	4.02	&	-1.83E-07	&	5.00	\\			
		&	1600&	3.67E-09	&	4.01	&	7.93E-09	&	4.02	&	2.87E-08	&	4.01	&	-5.72E-09	&	5.00	\\	\hline\noalign{\smallskip}		
RK(5,4,7)&	100	&	4.42E-05	&	-		&	9.47E-05	&	-		&	3.37E-04	&	-		&	-1.08E-04	&	-		\\			
		&	200	&	2.59E-06	&	4.09	&	5.57E-06	&	4.09	&	2.02E-05	&	4.06	&	-8.32E-07	&	7.01	\\			
		&	400	&	1.57E-07	&	4.05	&	3.38E-07	&	4.04	&	1.22E-06	&	4.05	&	-6.55E-09	&	6.99	\\			
		&	800	&	9.75E-09	&	4.01	&	2.10E-08	&	4.01	&	7.60E-08	&	4.00	&	-5.12E-11	&	7.00	\\			
		&	1600&	6.08E-10	&	4.00	&	1.31E-09	&	4.00	&	4.74E-09	&	4.00	&	-4.00E-13	&	7.00	\\	\hline\noalign{\smallskip}		
RK(6,4,9)&	100	&	1.53E-05	&	-		&	3.27E-05	&	-		&	1.20E-04	&	-		&	-9.68E-07	&	-		\\			
		&	200	&	9.51E-07	&	4.01	&	2.03E-06	&	4.01	&	7.38E-06	&	4.02	&	-1.86E-09	&	9.02	\\			
		&	400	&	5.86E-08	&	4.02	&	1.26E-07	&	4.01	&	4.54E-07	&	4.02	&	-3.66E-12	&	8.99	\\			
		&	800	&	3.65E-09	&	4.00	&	7.85E-09	&	4.00	&	2.84E-08	&	4.00	&	$\sim10^{-15}$	&	-	\\			
		&	1600&	2.28E-10	&	4.00	&	4.91E-10	&	4.00	&	1.77E-09	&	4.00	&	$\sim10^{-15}$	&	-	\\	\hline\noalign{\smallskip}		
RK(7,4,11)&	100	&	7.54E-06	&	-		&	1.61E-05	&	-		&	5.88E-05	&	-		&	-5.06E-09	&	-		\\			
		&	200	&	4.76E-07	&	3.99	&	1.02E-06	&	3.98	&	3.69E-06	&	3.99	&	-2.43E-12	&	11.02	\\			
		&	400	&	2.94E-08	&	4.02	&	6.33E-08	&	4.01	&	2.28E-07	&	4.02	&	-1.11E-15	&	11.10	\\			
		&	800	&	1.84E-09	&	4.00	&	3.95E-09	&	4.00	&	1.43E-08	&	4.00	&	$\sim10^{-15}$	&	-	\\			
		&	1600&	1.15E-10	&	4.00	&	2.47E-10	&	4.00	&	8.89E-10	&	4.01	&	$\sim10^{-15}$	&	-	\\	\hline		
\end{tabular}																														
\end{table}

\begin{table}										
		\begin{center}										
			\caption{CPU runtime comparison of four fourth-order RK methods achieving a target energy accuracy of 
				$|\epsilon_E| < 5 \times 10^{-10}$ for the peridynamic equation integrated up to $T=10$.}										
			\label{tab:PDcpu}										
			\begin{tabular}{lrrcr}										
				\hline \noalign{\smallskip}										
				method 		&	 $N_x$	&	$N_t$	&	$\epsilon_E$	&	CPU time (s) 		\\ \hline\noalign{\smallskip}
				RK(4,4,5)	&	3000	&	750		&	-4.94E-10		&	244.02		\\ \noalign{\smallskip}
				RK(5,4,7)	&	650		&	163		&	-4.40E-10		&	3.53		\\ \noalign{\smallskip}
				RK(6,4,9)	&	235		&	29		&	-4.32E-10		&	0.24		\\ \noalign{\smallskip}
				RK(7,4,11) 	&	125		&	16		&	-4.32E-10		&	0.10		\\ \hline
			\end{tabular}										
		\end{center}										
\end{table}										


 \subsection{One-dimensional Maxwell's equations}
 \label{sec:Maxwell}
 
Consider the one-dimensional Maxwell's equations 
\bea
\varepsilon_0 \pd{E_z}{t} &=& \pd{H_y}{x}, \\
\mu_0 \pd{H_y}{t} &=& \pd{E_z}{x}.
\eea
By employing a staggered grid and centered difference algorithm, we define $E^n_j = E_z(j\Delta x, n\Delta t)$ and $H^n_{j+1/2} = H_y((j+1/2)\Delta x, n\Delta t)$, yielding the semi-discrete equations:  
\bea
	\varepsilon_0 \pdt E_z(t_n,x_j) &=& \frac{H_{j+1/2}^n - H_{j-1/2}^n}{\Delta x},	\label{eq:semi_E} \\
	\mu_0 \pdt H_y(t_n,x_{j+1/2}) &=& \frac{E_{j+1}^n - E_j^n}{\Delta x}.			\label{eq:semi_H}
\eea
These equations can be written in matrix form as:
\begin{equation}
 \pdt \vtwo{\vec{E}}{\vec{H}} = \mtwo{1/\eps_0}{0}{0}{1/\mu_0} \mtwo{0}{C}{-C^\top}{0} \vtwo{\vec{E}}{\vec{H}},
\end{equation}
where the matrix $C$ represents the finite difference curl operator, and $\vec{E}^n=(E^n_0,E^n_1,...,E^n_{N_x})^\top$ and 
$\vec{H}^n=(H^n_{1/2},H^n_{3/2},...,H^n_{N_x+1/2})^\top$ are the solution vectors at time $n\Delta t$. 
For perfect electric conductor (PEC) boundaries, we have $E^n_0=E^n_N=0$. 
$H^n_{N+1/2}$ is outside the boundary and is set to zero. 
For example, when $N_x=4$, we have: 
\beq
C=\frac{1}{\Delta x}\pmat{ 	{ 1}&{ 0}&{ 0}&{ 0}&{0}\\
						{-1}&{ 1}&{ 0}&{ 0}&{0}\\
						{ 0}&{-1}&{ 1}&{ 0}&{0}\\
						{ 0}&{ 0}&{-1}&{ 1}&{0}\\
						{ 0}&{ 0}&{ 0}&{-1}&{1} }.
\eeq
The spectral radius of $C^\top C$ satisfies $\mathcal{\rho}(C^\top C)\le \frac{4}{\Delta x^2}$. 
Therefore, 
\beq
\left|\left| \mtwo{0}{C}{-C^\top}{0} \right|\right| = \sqrt{\mathcal{\rho}(C^\top C)} \le \frac{2}{\Delta x}.
\eeq
For the fourth-order RK methods, the corresponding stability condition becomes $\Delta t \le \frac{\lambda}{2c}\Delta x$, where $c=1/\sqrt{\epsilon_0\mu_0}$ is the speed of light and $\lambda$ is the stability limit provided in Table \ref{tab:p4}.

The total energy at $t=n\Delta t$ is given by
\beq \label{eq:energy_Maxwell}
\mathcal{E}_n = \frac{1}{2} \left(\varepsilon_0 ||\vec{E}^n||^2 + \mu_0 ||\vec{H}^n||^2 \right). 
\eeq							
Note that if we define the magnetic field at time $(n+1/2)\Delta t$, and discretize the equations (\ref{eq:semi_E})-(\ref{eq:semi_H}) using the leapfrog algorithm, we obtained the standard Yee Finite-Difference Time-Domain (FDTD) method \cite{Yee66,Taflove75,Taflove05}. 
Different from the energy defined in (\ref{eq:energy_Maxwell}), the FDTD method preserves the following numerical energy:  
\beq \label{eq:energy_FDTD}
\mathcal{\tilde{E}}_n = \frac{1}{2} \left(\varepsilon_0 ||\vec{E}^n||^2 + \mu_0 \innerproduct{\vec{H}^{n-1/2}}{\vec{H}^{n+1/2}} \right). 
\eeq	

In our simulations, the computational domain is $x\in[-5,5]$ with PEC boundary conditions.
The initial condition is a Gaussian pulse defined as 
\beq \label{eq:initial_pulse}
E(0,x) = \phi(x) = e^{-5x^2}\sin\left(\frac{2\pi}{\lambda_0} x\right), 
\eeq
with a carrier wavelength of $\lambda_0 = 0.2$.  
The final time is set to $T=10^{-8}$ to ensure that the pulse does not reach the boundaries. 
The exact solution is obtained using the d'Alembert formula for traveling waves: $E(t,x) = 0.5(\phi(x+ct) + \phi(x-ct))$. 

The simulation results are summarized in Table \ref{tab:Maxwell}. 
All four RK methods exhibit similar accuracy, with convergence rates of the solutions being second-order due to the spatial discretization being second-order. 
	Table \ref{tab:Maxwellcpu} compares the computational performance of four fourth-order RK methods 
	required to achieve a target energy accuracy of $|\epsilon_E| < 5 \times 10^{-10}$.
	As the number of stages in the RK method increases, the computational efficiency regarding energy accuracy improves dramatically. 
	Specifically, the seven-stage method RK(7,4,11) permits a 21 times reduction in the number of spatial nodes ($N_x$) 
	and a 30 times reduction in time steps ($N_t$) compared to the four-stage RK(4,4,5) method. 
	This results in a simulation speedup of over 240 times.



\begin{table}																								
\caption 	{Solution and energy convergence rates for 1D Maxwell's equations using fourth-order RK methods. }		
\label{tab:Maxwell}																									
\begin{tabular}																								
	{ c		c		c		c		c		c		c		c		c		c	}	\hline\noalign{\smallskip}				
	Method	&	$N_x$	&	$\epsilon_1$	&	Ratio	&	$\epsilon_2$	&	Ratio	&	$\epsilon_\infty$	&	Ratio	&	$\epsilon_E$	&	Ratio	\\	\hline\noalign{\smallskip}				
	RK(4,4,5)	&	2000	&	3.84E-03	&		&	1.06E-02	&		&	5.24E-02	&		&	-8.07E-04	&		\\					
	&	4000	&	9.46E-04	&	2.02	&	2.61E-03	&	2.02	&	1.29E-02	&	2.02	&	-2.55E-05	&	4.98	\\					
	&	8000	&	2.36E-04	&	2.01	&	6.49E-04	&	2.01	&	3.21E-03	&	2.01	&	-7.99E-07	&	5.00	\\					
	&	16000	&	5.89E-05	&	2.00	&	1.62E-04	&	2.00	&	8.03E-04	&	2.00	&	-2.50E-08	&	5.00	\\					
	&	32000	&	1.47E-05	&	2.00	&	4.05E-05	&	2.00	&	2.01E-04	&	2.00	&	-7.81E-10	&	5.00	\\	\hline\noalign{\smallskip}				
	RK(5,4,7)	&	2000	&	3.79E-03	&		&	1.04E-02	&		&	5.17E-02	&		&	-7.72E-06	&		\\					
	&	4000	&	9.43E-04	&	2.01	&	2.59E-03	&	2.01	&	1.29E-02	&	2.01	&	-6.11E-08	&	6.98	\\					
	&	8000	&	2.36E-04	&	2.00	&	6.48E-04	&	2.00	&	3.21E-03	&	2.00	&	-4.79E-10	&	6.99	\\					
	&	16000	&	5.88E-05	&	2.00	&	1.62E-04	&	2.00	&	8.02E-04	&	2.00	&	-3.74E-12	&	7.00	\\					
	&	32000	&	1.47E-05	&	2.00	&	4.05E-05	&	2.00	&	2.01E-04	&	2.00	&	-3.16E-14	&	6.89	\\	\hline\noalign{\smallskip}				
	RK(6,4,9)	&	2000	&	3.79E-03	&		&	1.04E-02	&		&	5.16E-02	&		&	-3.55E-08	&		\\					
	&	4000	&	9.42E-04	&	2.01	&	2.59E-03	&	2.01	&	1.28E-02	&	2.01	&	-7.03E-11	&	8.98	\\					
	&	8000	&	2.35E-04	&	2.00	&	6.48E-04	&	2.00	&	3.21E-03	&	2.00	&	-1.39E-13	&	8.99	\\					
	&	16000	&	5.88E-05	&	2.00	&	1.62E-04	&	2.00	&	8.02E-04	&	2.00	&	$\sim10^{-15}$	&	-	\\					
	&	32000	&	1.47E-05	&	2.00	&	4.05E-05	&	2.00	&	2.01E-04	&	2.00	&	$\sim10^{-15}$	&	-	\\	\hline\noalign{\smallskip}				
	RK(7,4,11)	&	2000	&	3.77E-03	&		&	1.04E-02	&		&	5.14E-02	&		&	-6.12E-11	&		\\					
	&	4000	&	9.42E-04	&	2.00	&	2.59E-03	&	2.00	&	1.29E-02	&	2.00	&	-3.01E-14	&	10.99	\\					
	&	8000	&	2.35E-04	&	2.00	&	6.48E-04	&	2.00	&	3.21E-03	&	2.00	&	$\sim10^{-15}$	&	-	\\					
	&	16000	&	5.88E-05	&	2.00	&	1.62E-04	&	2.00	&	8.02E-04	&	2.00	&	$\sim10^{-15}$	&	-	\\					
	&	32000	&	1.47E-05	&	2.00	&	4.05E-05	&	2.00	&	2.01E-04	&	2.00	&	$\sim10^{-15}$	&	-	\\	\hline				
\end{tabular}																								
\end{table}																									

\begin{table}																			
		\begin{center}										
			\caption{CPU runtime comparison of four fourth-order RK methods achieving a target energy accuracy of 
				$|\epsilon_E| < 5\times10^{-10}$ for Maxwell's equations integrated up to $T=10^{-8}$.}										
			\label{tab:Maxwellcpu}										
			\begin{tabular}{lrrcc}										
				\hline \noalign{\smallskip}										
				method 		&	 $N_x$	&	$N_t$	&$\epsilon_E$	&	CPU time (s) 	\\ \hline\noalign{\smallskip}
				RK(4,4,5)	&	35000	&	7419	&	-4.99E-10	&	8.45		\\ \noalign{\smallskip}
				RK(5,4,7)	&	8000	&	1385	&	-4.79E-10	&	0.56		\\ \noalign{\smallskip}
				RK(6,4,9)	&	3240	&	502		&	-4.68E-10	&	0.14		\\ \noalign{\smallskip}
				RK(7,4,11) 	&	1660	&	249		&	-4.70E-10	&	0.03		\\ \hline
			\end{tabular}										
		\end{center}																		
\end{table}


\section{ Nonlinear examples}
\label{sec:nonlinear_examples}

In this section, we first evaluate the performance of the proposed RK325 method using a variety of skew-symmetric systems, including 
Euler's equations for rigid body dynamics, the nonlinear Schr\"odinger, the KdV, the Landau--Lifshitz, and Burgers' equations.
We compare the RK325 method with RK4 and the third-order strong stability-preserving RK (SSPRK3) method \cite{gottlieb1998total,gottlieb2001strong}.

Next, we test the proposed RK427 and RK547 methods, which feature seventh-order energy accuracy, 
using a nonlinear Hamiltonian oscillator with cubic nonlinearity and an amplitude-dependent frequency. 
In the final example, we apply the RK547 method to the Maxwell--Kerr system and compare it with RK4.

\subsection{Euler's equations for rigid body dynamics}
In the first example, we consider Euler's equations for rigid body dynamics: 
\begin{align}
	I_1 w_1' &= \left( I_2 - I_3 \right) w_2 w_3 + M_1, \\
	I_2 w_2' &= \left( I_3 - I_1 \right) w_3 w_1 + M_2, \\
	I_3 w_3' &= \left( I_1 - I_2 \right) w_1 w_2 + M_3,
\end{align}
where $w=(w_1,w_2,w_3)$, $I=(I_1,I_2,I_3)$, and $M=(M_1,M_2,M_3)$ represent 
the angular velocity, the principal moment of inertia, and the external torque, respectively.
In our simulations, we consider the case of no torque ($M=0$), 
so that the angular momentum magnitude $L$ and the rotational kinetic energy $E$ are conserved quantities:
\begin{equation}
	L = \sqrt{(I_1 w_1)^2+(I_2 w_2)^2+(I_3 w_3)^2},
\end{equation}
\begin{equation}
	E = \frac{1}{2}(I_1 w_1^2+I_2 w_2^2+I_3 w_3^2).
\end{equation}

Let $I=(0.5,1,2)$, $t\in[0,10]$, and initial condition to be $w(0)=(1,1,1)$. 
The solution obtained using RK4 with $\Delta t = 0.005$ is used as the reference solution when computing solution convergence rates. 
Momentum and energy errors are calculated as $\epsilon_L= |L(T)-L(0)|$ and $\epsilon_E=|E(T)-E(0)|$, where $T$ is the final time.
Simulation results are shown in Table \ref{tab:Euler} and Fig. \ref{fig:Euler}, 
which demonstrate that the RK325 method achieves fifth-order for both momentum and energy, the same as RK4.

\begin{table}	
	\begin{center}														
		\caption{Solution, momentum, and energy convergence rates for	Euler's equations for rigid body dynamics.		}								
		\label	{tab:Euler}														
		\begin{tabular}{cccccccc}															
			\hline \noalign{\smallskip}															
			Method	&	$\Delta t$	&	$\epsilon_2$	&	Rate	&	$\epsilon_L$	&	Rate	&	$\epsilon_E$	&	Rate	\\
			\hline \noalign{\smallskip}															
	SSPRK3	&	0.400	&	3.05E-01	&		&	9.01E-02	&		&	2.46E-01	&		\\
			&	0.200	&	4.17E-02	&	2.87	&	1.53E-02	&	2.56	&	4.19E-02	&	2.55	\\
			&	0.100	&	5.37E-03	&	2.96	&	2.03E-03	&	2.91	&	5.56E-03	&	2.91	\\
			&	0.050	&	6.79E-04	&	2.98	&	2.57E-04	&	2.98	&	7.04E-04	&	2.98	\\
			&	0.025	&	8.54E-05	&	2.99	&	3.22E-05	&	3.00	&	8.82E-05	&	3.00	\\
			\hline \noalign{\smallskip}															
		RK4	&	0.400	&	5.54E-02	&		&	1.09E-02	&		&	2.90E-02	&		\\
			&	0.200	&	3.37E-03	&	4.04	&	3.64E-04	&	4.91	&	9.83E-04	&	4.88	\\
			&	0.100	&	2.03E-04	&	4.05	&	1.10E-05	&	5.05	&	3.06E-05	&	5.00	\\
			&	0.050	&	1.24E-05	&	4.03	&	3.10E-07	&	5.15	&	9.20E-07	&	5.06	\\
			&	0.025	&	7.64E-07	&	4.02	&	7.49E-09	&	5.37	&	2.62E-08	&	5.13	\\
			\hline \noalign{\smallskip}															
	RK325	&	0.400	&	5.94E-01	&		&	1.78E-02	&		&	4.32E-02	&		\\
			&	0.200	&	1.29E-01	&	2.21	&	5.38E-04	&	5.05	&	1.31E-03	&	5.05	\\
			&	0.100	&	3.11E-02	&	2.05	&	1.68E-05	&	5.00	&	4.09E-05	&	5.00	\\
			&	0.050	&	7.71E-03	&	2.01	&	5.27E-07	&	5.00	&	1.28E-06	&	5.00	\\
			&	0.025	&	1.92E-03	&	2.00	&	1.65E-08	&	5.00	&	4.00E-08	&	5.00	\\
			\hline \noalign{\smallskip}															
		\end{tabular}	
	\end{center}												
\end{table}															

\begin{figure}[h]
	\begin{center}
		\subfigure[]{ \includegraphics[width=0.45\textwidth]{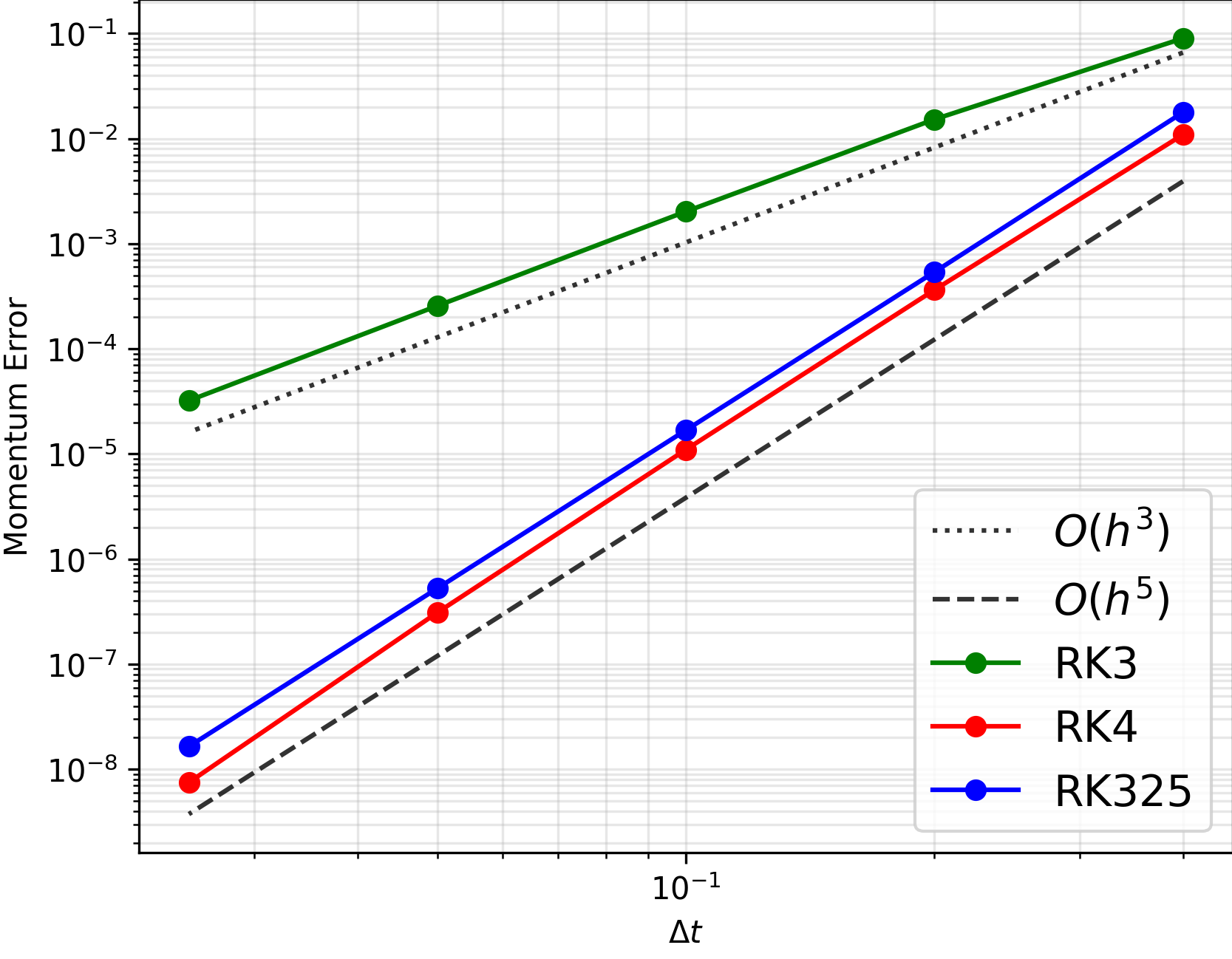} }
		\subfigure[]{ \includegraphics[width=0.45\textwidth]{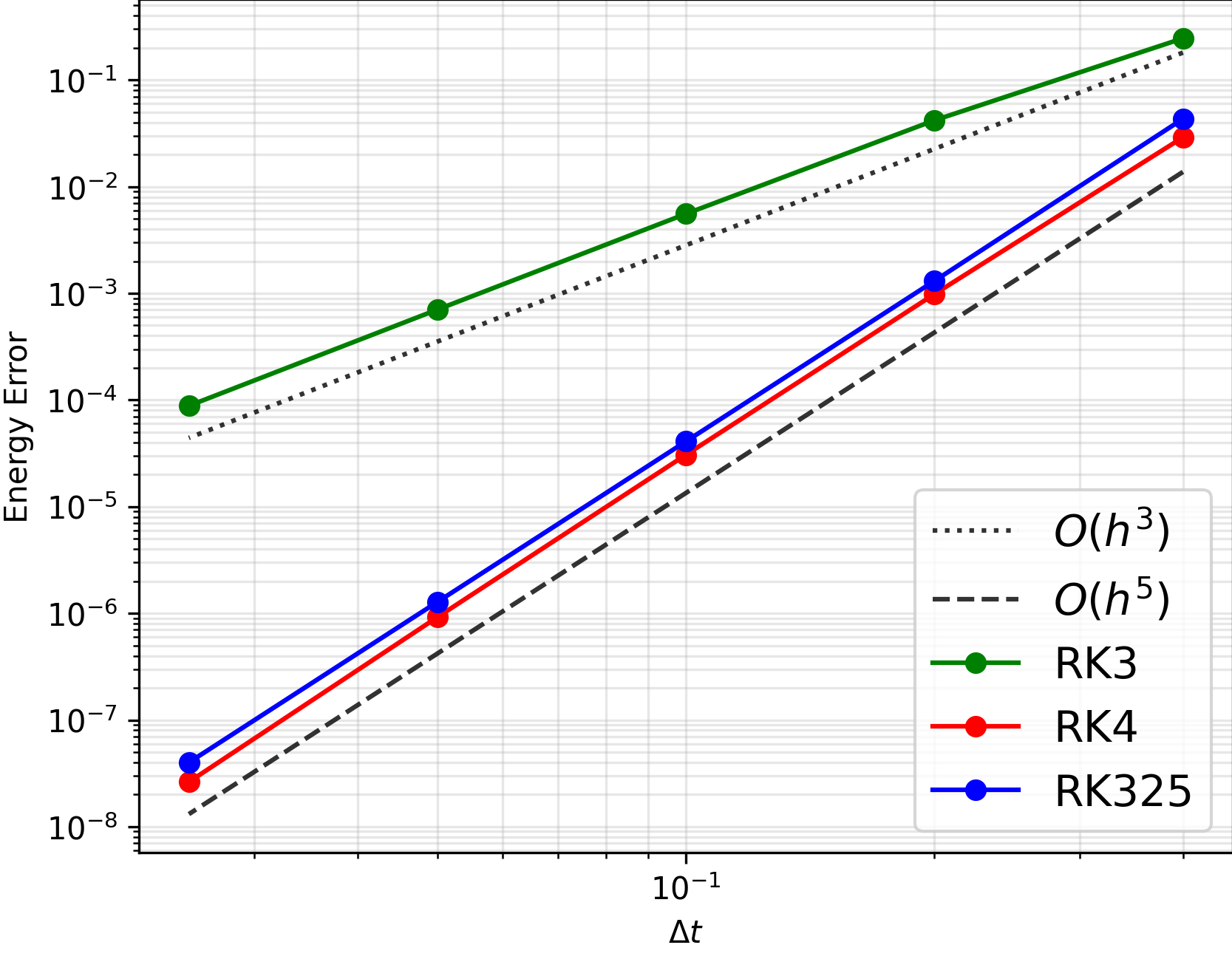} }
		\caption{Comparison of (a) momentum and (b) energy convergence rates for Euler's equation for rigid body dynamics.}
		\label{fig:Euler}
	\end{center}
\end{figure}

\subsection{Nonlinear Schr\"odinger equation}

In this example, we consider the nonlinear Schr\"odinger equation (NLSE) 
\begin{equation}
	i \phi_t = -\phi_{xx} - |\phi|^2 \phi,
\end{equation}
with initial condition $\phi(t=0,x)=\text{sech}(x)$.
The mass is given by
\begin{equation}
	m(t) = \int_a^b{|\phi(t,x)|^2 dx} 
\end{equation}
and the discrete mass is computed as 
\begin{equation}
	m(t) = \sum_j{|\phi_j^n|^2 \Delta x}, 
\end{equation}
where $\phi_j^n$ represents the numerical solution at $t_n=n \Delta t$ and $x=x_j$.
The computational domain is $x \in [-20,20]$ and $t \in [0,5]$.
The term $\phi_{xx}$ is discretized using standard centered finite difference with periodic boundary conditions.
The spatial mesh size is fixed at $N_x=100$, so $\Delta x=0.4$.
The solution obtained using RK4 with $\Delta t = 1\times10^{-4}$ is used as the reference solution when computing convergence rates.
Mass error is computed as $\epsilon_m = |m(T) - m(0)|$.
Simulation results are shown in Table \ref{tab:NLSE} and Fig.~\ref{fig:NLSE},  
which demonstrate that the RK325 method achieves fifth-order energy accuracy, the same as RK4.

\begin{table}	
	\begin{center}																				
		\caption{Solution and mass convergence rates of NLSE.}																				
		\label{tab:NLSE}																				
		\begin{tabular}{cccccccccc}																				
			\hline \noalign{\smallskip}																				
			Method	&	$\Delta t$	&	$\epsilon_1$	&	Rate	&	$\epsilon_2$	&	Rate	&	$\epsilon_\infty$	&	Rate	&	$\epsilon_m$	&	Rate	\\	
			\hline \noalign{\smallskip}																				
	SSPRK3	&	0.0400	&	2.23E-03	&		&	5.94E-04	&		&	2.39E-04	&		&	2.34E-05	&		\\	
			&	0.0200	&	3.19E-04	&	2.80	&	8.12E-05	&	2.87	&	3.28E-05	&	2.86	&	3.03E-06	&	2.95	\\	
			&	0.0100	&	4.12E-05	&	2.95	&	1.03E-05	&	2.98	&	4.22E-06	&	2.96	&	3.82E-07	&	2.99	\\	
			&	0.0050	&	5.18E-06	&	2.99	&	1.29E-06	&	3.00	&	5.32E-07	&	2.99	&	4.78E-08	&	3.00	\\	
			&	0.0025	&	6.48E-07	&	3.00	&	1.62E-07	&	3.00	&	6.66E-08	&	3.00	&	5.97E-09	&	3.00	\\	
			\hline \noalign{\smallskip}																				
		RK4	&	0.0400	&	2.40E-04	&		&	5.80E-05	&		&	2.64E-05	&		&	2.15E-07	&		\\	
			&	0.0200	&	1.61E-05	&	3.90	&	3.74E-06	&	3.95	&	1.75E-06	&	3.92	&	6.66E-09	&	5.02	\\	
			&	0.0100	&	1.01E-06	&	3.99	&	2.34E-07	&	4.00	&	1.10E-07	&	3.99	&	1.97E-10	&	5.08	\\	
			&	0.0050	&	6.32E-08	&	4.00	&	1.46E-08	&	4.00	&	6.87E-09	&	4.00	&	5.40E-12	&	5.19	\\	
			&	0.0025	&	3.95E-09	&	4.00	&	9.16E-10	&	4.00	&	4.30E-10	&	4.00	&	1.22E-13	&	5.47	\\	
			\hline \noalign{\smallskip}																				
	RK325	&	0.0400	&	7.49E-03	&		&	1.93E-03	&		&	7.93E-04	&		&	2.57E-07	&		\\	
			&	0.0200	&	1.82E-03	&	2.04	&	4.74E-04	&	2.03	&	1.93E-04	&	2.04	&	7.96E-09	&	5.01	\\	
			&	0.0100	&	4.49E-04	&	2.02	&	1.18E-04	&	2.01	&	4.83E-05	&	2.00	&	2.49E-10	&	5.00	\\	
			&	0.0050	&	1.12E-04	&	2.00	&	2.94E-05	&	2.00	&	1.21E-05	&	2.00	&	7.77E-12	&	5.00	\\	
			&	0.0025	&	2.80E-05	&	2.00	&	7.34E-06	&	2.00	&	3.01E-06	&	2.00	&	2.42E-13	&	5.00	\\	
			\hline \noalign{\smallskip}																				
		\end{tabular}
	\end{center}																			
\end{table}																

\begin{figure}[h]
	\begin{center}
		\includegraphics[width=0.5\textwidth]{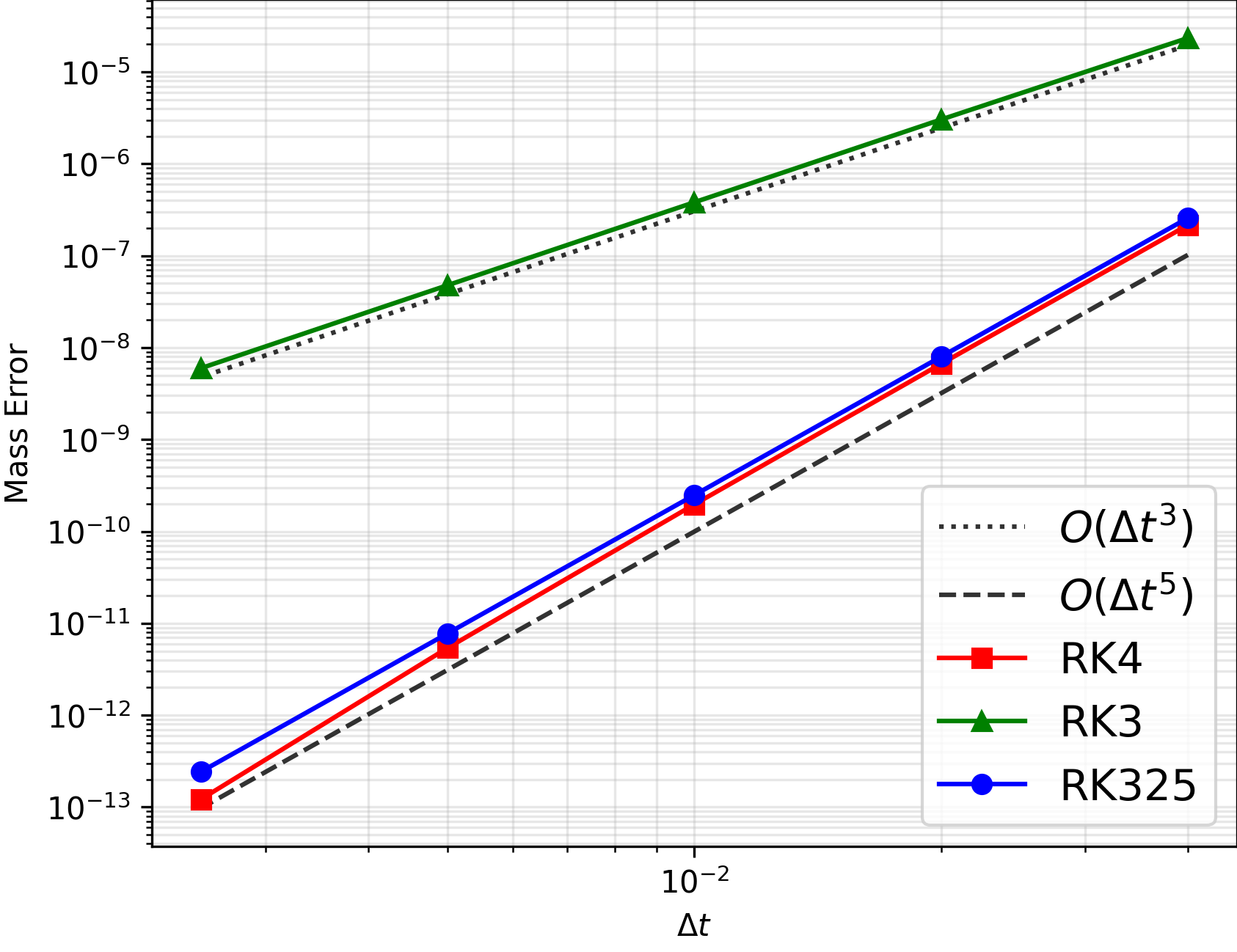} 	
		\caption{Comparison of mass convergence rates for NLSE.}
		\label{fig:NLSE}
	\end{center}
\end{figure}

\subsection{KdV equation}

In this example, we consider the Korteweg--De Vries (KdV) equation
\begin{equation}
	u_t + 6 u u_x + u_{xxx}= 0, 
\end{equation}
with initial condition $u(t=0,x)=2sech^2(x+4)$.
We use the skew-symmetric discretization for the convective term:
\begin{equation}
	6 u u_x = 2 uu_x + 2 (u^2)_x,
\end{equation}
which leads to the following semi-discretized equation:
\begin{equation}
	\frac{\partial u_j}{\partial t} = -2 \left( u_j \frac{u_{j+1} - u_{j-1}}{2 \Delta x} + \frac{u_{j+1}^2 - u_{j-1}^2}{2 \Delta x} \right)
	- \frac{u_{j+2} - 2 u_{j+1} + 2 u_{j-1} - u_{j-2}}{2 \Delta x^3}.
\end{equation}
The computational domain is $x\in[-10,10]$ with periodic boundary conditions. 
The mesh size is fixed at $N_x=50$. 
The final time is 1.
The energy is computed as $\mathcal{E} = \frac{1}{2}\sum{u_j^2}$.
The solution obtained using RK4 with $\Delta t = 0.0005$ is used as the reference solution when computing convergence rates.
Simulation results are shown in Table \ref{tab:KdV} and Fig. \ref{fig:KdV}, 
which demonstrate that the RK325 method achieves fifth-order energy accuracy.

\begin{table}	
	\begin{center}																			
		\caption{Solution and energy convergence rates of 					KdV equation. }												
		\label	{tab:KdV}																													
		\begin{tabular}{cccccccccc}																			
			\hline \noalign{\smallskip}																			
			Method	&	$\Delta t$	&	$\epsilon_1$	&	Rate	&	$\epsilon_2$	&	Rate	&	$\epsilon_\infty$	&	Rate	&	$\epsilon_E$	&	Rate	\\
			\hline \noalign{\smallskip}																			
	SSPRK3	&	0.001000	&	1.87E-04	&		&	5.09E-05	&		&	2.82E-05	&		&	2.67E-07	&		\\
			&	0.000500	&	8.72E-05	&	1.10	&	2.40E-05	&	1.08	&	1.58E-05	&	0.83	&	3.47E-08	&	2.94	\\
			&	0.000250	&	1.51E-05	&	2.53	&	4.22E-06	&	2.51	&	2.78E-06	&	2.51	&	4.49E-09	&	2.95	\\
			&	0.000125	&	1.97E-06	&	2.94	&	5.54E-07	&	2.93	&	3.64E-07	&	2.93	&	5.65E-10	&	2.99	\\
			
			\hline \noalign{\smallskip}																			
		RK4	&	0.001000	&	9.99E-05	&		&	2.79E-05	&		&	1.80E-05	&		&	1.05E-09	&		\\
			&	0.000500	&	7.87E-06	&	3.67	&	2.27E-06	&	3.62	&	1.45E-06	&	3.63	&	4.98E-11	&	4.39	\\
			&	0.000250	&	4.96E-07	&	3.99	&	1.43E-07	&	3.99	&	8.66E-08	&	4.07	&	1.60E-12	&	4.96	\\
			&	0.000125	&	3.10E-08	&	4.00	&	8.91E-09	&	4.00	&	5.22E-09	&	4.05	&	9.41E-14	&	4.08	\\
			
			\hline \noalign{\smallskip}																			
	RK325	&	0.001000	&	4.22E-04	&		&	1.19E-04	&		&	6.44E-05	&		&	3.44E-09	&		\\
			&	0.000500	&	3.09E-04	&	0.45	&	8.46E-05	&	0.50	&	5.07E-05	&	0.35	&	5.86E-11	&	5.88	\\
			&	0.000250	&	1.02E-04	&	1.60	&	2.81E-05	&	1.59	&	1.66E-05	&	1.61	&	1.81E-12	&	5.02	\\
			&	0.000125	&	2.58E-05	&	1.99	&	7.13E-06	&	1.98	&	3.97E-06	&	2.07	&	5.68E-14	&	4.99	\\
			
			\hline \noalign{\smallskip}																			
		\end{tabular}
	\end{center}																			
\end{table}																			

\begin{figure}[h]	
	\begin{center}
		\includegraphics[width=0.5\textwidth]{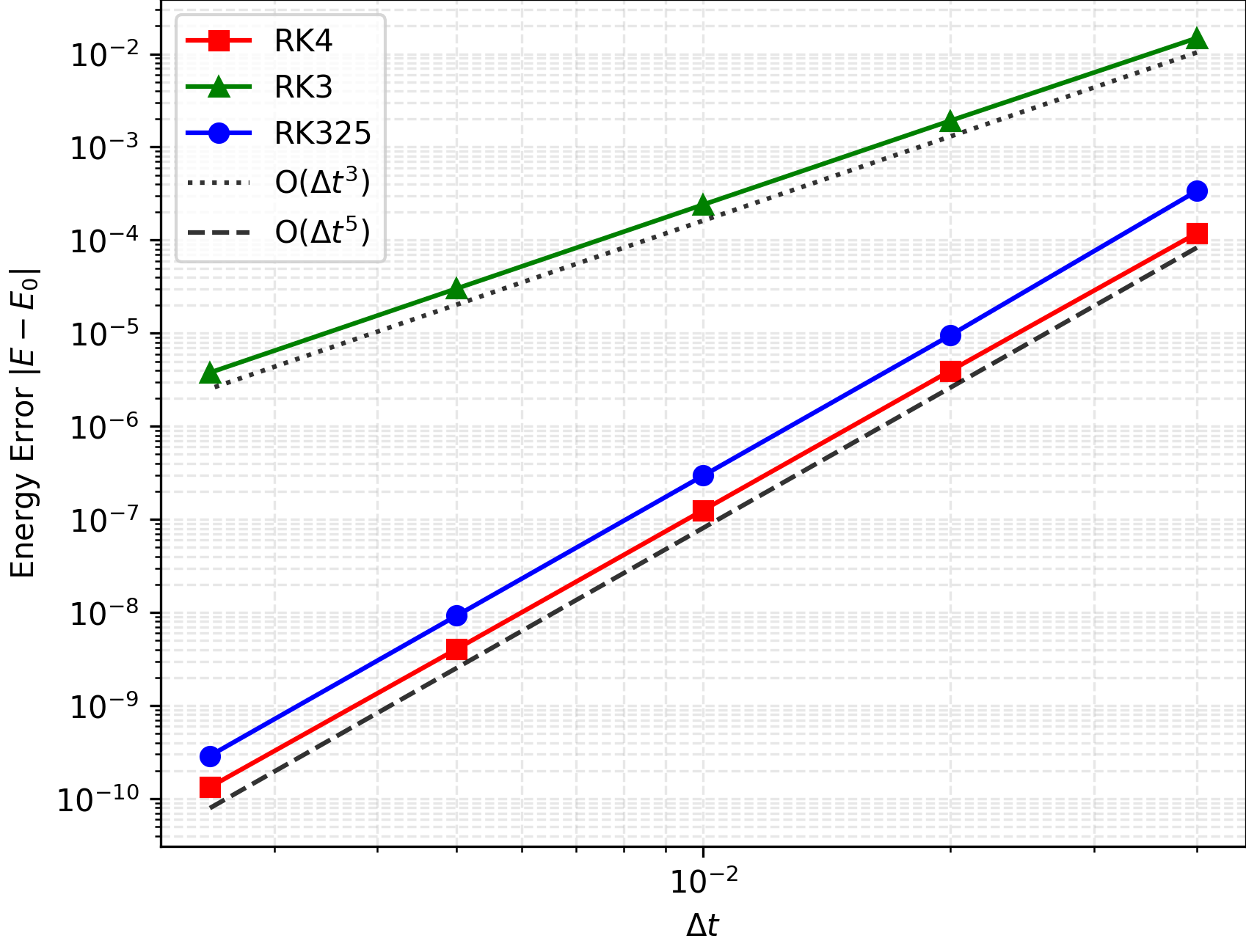} 	
		\caption{Solution and energy convergence rates for KdV equation.}
		\label{fig:KdV}
	\end{center}
\end{figure}

\subsection{Burgers' equation}

In this example, we consider inviscid Burgers' equation 
\begin{equation}
	u_t + u u_x = 0,
\end{equation}
with initial condition $u(t=0,x)=sin(x)$.
We use the skew-symmetric discretization for the convective term:
\begin{equation}
	u u_x = \frac{1}{3} uu_x + \frac{1}{3}(u^2)_x,
\end{equation}
which leads to the following semi-discretized equation:
\begin{equation}
	\frac{\partial u_j}{\partial t} = -\frac{1}{3} \left( u_j \frac{u_{j+1} - u_{j-1}}{2 \Delta x} + \frac{u_{j+1}^2 - u_{j-1}^2}{2 \Delta x} \right).
\end{equation}
The energy is computed as $\mathcal{E} = \frac{1}{2}\sum{u_j^2}$.
The computational domain is $x\in[0,2\pi]$ with periodic boundary conditions, 
and the simulation is performed for $0\le t \le 1$ using a spatial mesh size of $N_x=200$. 
The solution obtained using RK4 with $\Delta t = 1\times10^{-4}$ serves as the reference solution for computing solution convergence rates.
Simulation results are presented in Table \ref{tab:Burgers} and Fig. \ref{fig:Burgers}. 
It can be observed that the RK325 method maintains fifth-order energy accuracy, whereas the accuracy of RK4 is reduced to fourth order.

\begin{table}																			
	\begin{center}																			
		\caption{Solution and energy convergence rates of 					Burgers equation.		}												
		\label	{tab:Burgers}																		
		\begin{tabular}{cccccccccc}																			
			\hline \noalign{\smallskip}																			
			Method	&	$\Delta t$	&	$\epsilon_1$	&	Rate	&	$\epsilon_2$	&	Rate	&	$\epsilon_\infty$	&	Rate	&	$\epsilon_E$	&	Rate	\\
			\hline \noalign{\smallskip}																			
		RK3	&	0.02000	&	1.97E-05	&		&	3.56E-05	&		&	1.07E-04	&		&	1.60E-07	&		\\
			&	0.01000	&	2.39E-06	&	3.04	&	4.34E-06	&	3.03	&	1.31E-05	&	3.04	&	2.04E-08	&	2.98	\\
			&	0.00500	&	2.91E-07	&	3.03	&	5.35E-07	&	3.02	&	1.60E-06	&	3.03	&	2.47E-09	&	3.04	\\
			&	0.00250	&	3.60E-08	&	3.02	&	6.64E-08	&	3.01	&	1.96E-07	&	3.02	&	3.01E-10	&	3.04	\\
			&	0.00125	&	4.47E-09	&	3.01	&	8.27E-09	&	3.01	&	2.44E-08	&	3.01	&	3.70E-11	&	3.02	\\
			\hline \noalign{\smallskip}																			
		RK4	&	0.02000	&	1.13E-06	&		&	2.57E-06	&		&	8.58E-06	&		&	8.80E-08	&		\\
			&	0.01000	&	7.51E-08	&	3.92	&	1.61E-07	&	4.00	&	5.26E-07	&	4.03	&	5.11E-09	&	4.11	\\
			&	0.00500	&	4.88E-09	&	3.94	&	1.00E-08	&	4.00	&	3.23E-08	&	4.03	&	3.06E-10	&	4.06	\\
			&	0.00250	&	3.10E-10	&	3.98	&	6.26E-10	&	4.00	&	1.99E-09	&	4.02	&	1.87E-11	&	4.03	\\
			&	0.00125	&	1.96E-11	&	3.98	&	3.91E-11	&	4.00	&	1.24E-10	&	4.01	&	1.16E-12	&	4.02	\\
			\hline \noalign{\smallskip}																			
	RK325	&	0.02000	&	1.18E-04	&		&	1.95E-04	&		&	7.56E-04	&		&	2.71E-08	&		\\
			&	0.01000	&	2.93E-05	&	2.01	&	4.81E-05	&	2.02	&	1.86E-04	&	2.02	&	8.60E-10	&	4.98	\\
			&	0.00500	&	7.31E-06	&	2.00	&	1.20E-05	&	2.00	&	4.65E-05	&	2.00	&	2.71E-11	&	4.99	\\
			&	0.00250	&	1.83E-06	&	2.00	&	2.99E-06	&	2.00	&	1.16E-05	&	2.00	&	8.48E-13	&	5.00	\\
			&	0.00125	&	4.56E-07	&	2.00	&	7.48E-07	&	2.00	&	2.90E-06	&	2.00	&	2.66E-14	&	4.99	\\
			\hline \noalign{\smallskip}																			
		\end{tabular}																			
	\end{center}																		
\end{table}																			

\begin{figure}[h]
	\begin{center}
		\includegraphics[width=0.5\textwidth]{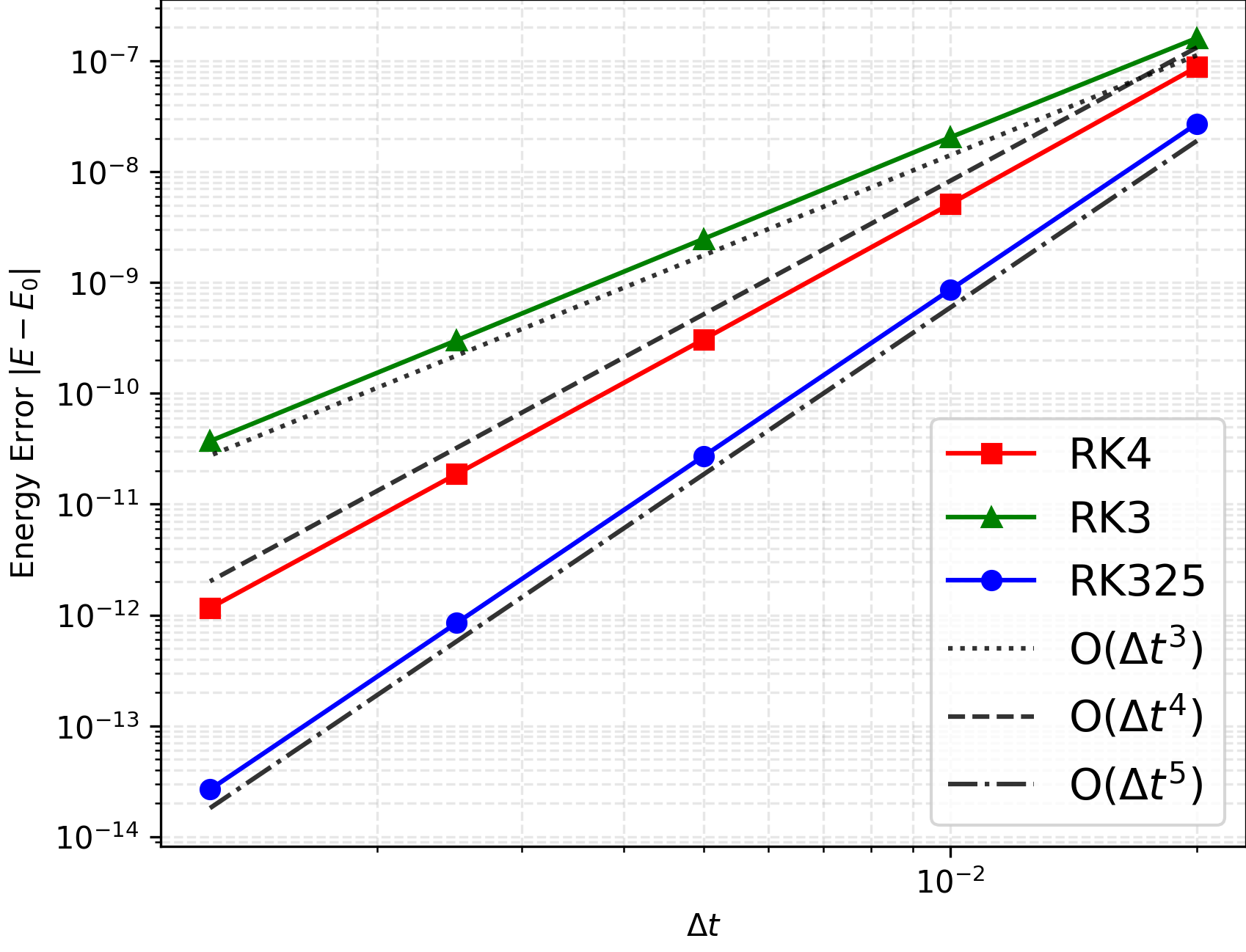} 	
		\caption{Solution and energy convergence rates for Burgers equation.}
		\label{fig:Burgers}
	\end{center}
\end{figure}

\subsection{Landau--Lifshitz equation}																			

In this example, we consider the Landau--Lifshitz (LL) equation
\begin{equation}
	\frac{dM}{dt} = -\gamma M\times H_{eff} - \frac{\alpha \gamma}{M_s} M\times(M\times H_{eff}),
\end{equation}
where $M$, $M_s$, $\gamma$, $\alpha$, and $H_{eff}$ denote the magnetization vector, the saturation magnetization, 
the gyromagnetic ratio, the damping constant, and the effective magnetic field, respectively.
The magnetization magnitude $|M|$ is a conserved quantity. Here, we test the convergence rate of $|M|^2$. 

In our simulations, we set the parameters to $\gamma=1$, $\alpha=0.1$, $M_s=1$, and $H_{eff}=(0,0,1)$, 
with an initial condition of $M(0)=(1, 0.1, 0)$. 
The simulations run until $t=5$. 
The solution obtained using RK4 with $\Delta t=1/2560$ serves as the reference solution for computing convergence rates.
Simulation results are presented in Table \ref{tab:LL} and Fig. \ref{fig:LL}. 
Similar to the previous Burgers' equation example, the RK325 method maintains fifth-order energy accuracy, 
whereas RK4 drops to fourth order.

\begin{table}																			
	\begin{center}																			
		\caption{Convergence rates of solution and $|M|^2$ for Landau--Lifshitz equation.}												
		\label	{tab:LL}																		
		\begin{tabular}{cccccccccc}																			
			\hline \noalign{\smallskip}																			
			Method	&	$\Delta t$	&	$\epsilon_1$	&	Rate	&	$\epsilon_2$	&	Rate	&	$\epsilon_\infty$	&	Rate	&	$\epsilon(|M|^2)$	&	Rate	\\
			\hline \noalign{\smallskip}																			
	SSPRK3	&	0.10000	&	3.24E-04	&		&	2.00E-04	&		&	1.64E-04	&		&	3.99E-04	&		\\
			&	0.05000	&	4.11E-05	&	2.98	&	2.50E-05	&	3.00	&	2.00E-05	&	3.03	&	4.99E-05	&	3.00	\\
			&	0.02500	&	5.17E-06	&	2.99	&	3.13E-06	&	3.00	&	2.47E-06	&	3.02	&	6.24E-06	&	3.00	\\
			&	0.01250	&	6.49E-07	&	3.00	&	3.91E-07	&	3.00	&	3.07E-07	&	3.01	&	7.80E-07	&	3.00	\\
			&	0.00625	&	8.12E-08	&	3.00	&	4.89E-08	&	3.00	&	3.83E-08	&	3.00	&	9.74E-08	&	3.00	\\
			\hline \noalign{\smallskip}																			
		RK4	&	0.10000	&	5.93E-06	&		&	3.83E-06	&		&	3.29E-06	&		&	2.14E-07	&		\\
			&	0.05000	&	3.70E-07	&	4.00	&	2.38E-07	&	4.00	&	2.01E-07	&	4.03	&	7.24E-09	&	4.89	\\
			&	0.02500	&	2.31E-08	&	4.00	&	1.49E-08	&	4.00	&	1.25E-08	&	4.02	&	1.10E-09	&	2.72	\\
			&	0.01250	&	1.44E-09	&	4.00	&	9.29E-10	&	4.00	&	7.74E-10	&	4.01	&	8.87E-11	&	3.63	\\
			&	0.00625	&	9.00E-11	&	4.00	&	5.80E-11	&	4.00	&	4.82E-11	&	4.00	&	6.17E-12	&	3.85	\\
			\hline \noalign{\smallskip}																			
	RK325	&	0.10000	&	3.05E-03	&		&	1.93E-03	&		&	1.61E-03	&		&	7.59E-07	&		\\
			&	0.05000	&	7.60E-04	&	2.00	&	4.82E-04	&	2.00	&	4.03E-04	&	2.00	&	2.37E-08	&	5.00	\\
			&	0.02500	&	1.90E-04	&	2.00	&	1.20E-04	&	2.00	&	1.01E-04	&	2.00	&	7.40E-10	&	5.00	\\
			&	0.01250	&	4.75E-05	&	2.00	&	3.01E-05	&	2.00	&	2.51E-05	&	2.00	&	2.31E-11	&	5.00	\\
			&	0.00625	&	1.19E-05	&	2.00	&	7.52E-06	&	2.00	&	6.29E-06	&	2.00	&	7.23E-13	&	5.00	\\
			\hline \noalign{\smallskip}																			
		\end{tabular}																			
	\end{center}																	
\end{table}																			

\begin{figure}
	\begin{center}
		\includegraphics[width=0.5\textwidth]{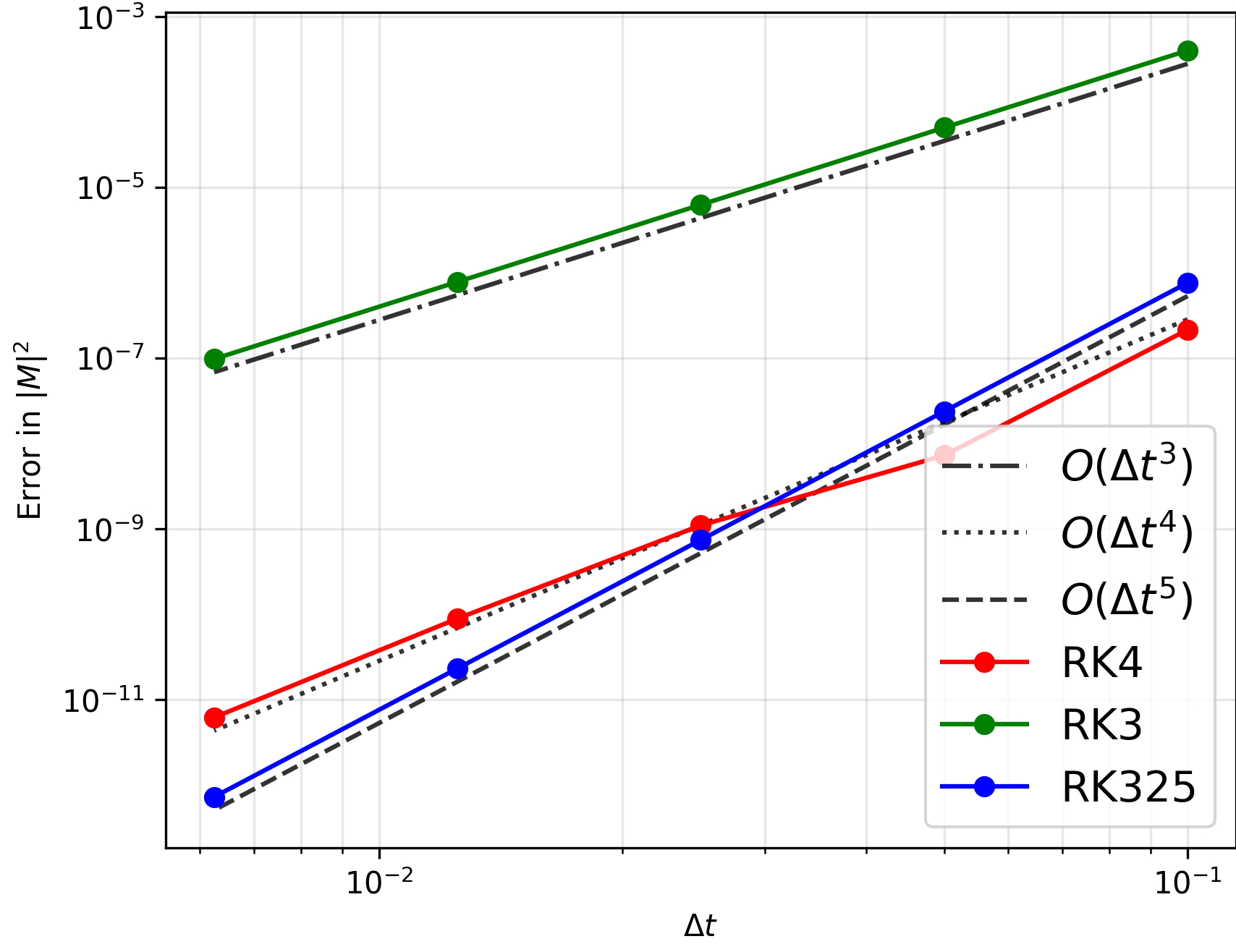} 	
		\caption{Convergence rates of $|M|^2$ for Landau--Lifshitz equation.}
		\label{fig:LL}
	\end{center}
\end{figure}

\subsection{Nonlinear Hamiltonian oscillator with cubic nonlinearity}

In this example, we examine the performance of the proposed seventh-order energy-accurate methods, specifically the RK427 and RK547 methods. 
We consider a nonlinear Hamiltonian oscillator with cubic nonlinearity \cite{guckenheimer2013nonlinear}
\beq
\frac{d}{dt}\vtwo{u}{v} = \mtwo{~0}{~\omega(u,v)}{-\omega(u,v)}{~0} \vtwo{u}{v},
\eeq
where $\omega(u,v)=1 + \frac{1}{2}(u^2+v^2)$, with initial conditions $u(0)=1$, $v(0)=0$, and $0\le t\le 10$.
Energy is given by $\mathcal{E} = u^2+v^2$.
The solution obtained using RK4 with $\Delta t=0.001$ serves as the reference solution for computing convergence rates.

As shown in Table \ref{tab:nho}, the simulation results demonstrate the solution and energy convergence rates of the proposed RK methods. 
In particular, the RK547 method yields a solution error similar to that of RK4, while it outperforms RK4 in terms of energy accuracy.

\begin{table}	
	\begin{center}										
	\caption{Solution and energy convergence rates of nonlinear harmonic oscillators using RK4, RK325, RK427a, RK427b, and RK547.
			$\eps_2$ and $\eps_E$ are the $L_2$ error norm and energy deviation, respectively. }				
	\label	{tab:nho}										
	\begin{tabular}		{llcccc}									
		\hline \noalign{\smallskip}											
Method	&	$\Delta t$	&	$\eps_2$	&	Rate	&	$\eps_E$	&	Rate	\\
		\hline \noalign{\smallskip}											
	RK4	&	0.500000	&	1.82E-01	&			&	4.35E-02	&		\\
		&	0.250000	&	1.01E-02	&	4.17	&	1.37E-03	&	4.99	\\
		&	0.125000	&	5.49E-04	&	4.20	&	4.10E-05	&	5.06	\\
		&	0.062500	&	3.15E-05	&	4.12	&	1.26E-06	&	5.02	\\
		&	0.031250	&	1.87E-06	&	4.07	&	3.93E-08	&	5.00	\\
		&	0.015625	&	1.14E-07	&	4.04	&	1.23E-09	&	5.00	\\
		\hline \noalign{\smallskip}											
RK325	&	0.500000	&	1.22E+00	&			&	1.09E-01	&		\\
		&	0.250000	&	1.98E-01	&	2.62	&	3.02E-03	&	5.17	\\
		&	0.125000	&	4.51E-02	&	2.13	&	9.59E-05	&	4.98	\\
		&	0.062500	&	1.11E-02	&	2.02	&	3.01E-06	&	4.99	\\
		&	0.031250	&	2.75E-03	&	2.01	&	9.43E-08	&	5.00	\\
		&	0.015625	&	6.87E-04	&	2.00	&	2.95E-09	&	5.00	\\
		\hline \noalign{\smallskip}											
RK427a	&	0.500000	&	2.92E-01	&			&	8.45E-02	&		\\
		&	0.250000	&	8.70E-02	&	1.75	&	7.89E-04	&	6.74	\\
		&	0.125000	&	1.91E-02	&	2.19	&	6.12E-06	&	7.01	\\
		&	0.062500	&	4.58E-03	&	2.06	&	4.77E-08	&	7.00	\\
		&	0.031250	&	1.13E-03	&	2.02	&	3.72E-10	&	7.00	\\
		&	0.015625	&	2.82E-04	&	2.00	&	2.91E-12	&	7.00	\\
		\hline \noalign{\smallskip}											
RK427b	&	0.500000	&	3.93E-01	&			&	2.39E-03	&		\\
		&	0.250000	&	8.81E-02	&	2.16	&	1.74E-05	&	7.10	\\
		&	0.125000	&	2.16E-02	&	2.03	&	1.33E-07	&	7.03	\\
		&	0.062500	&	5.37E-03	&	2.01	&	1.03E-09	&	7.01	\\
		&	0.031250	&	1.34E-03	&	2.00	&	8.05E-12	&	7.00	\\
		&	0.015625	&	3.35E-04	&	2.00	&	6.24E-14	&	7.01	\\
		\hline \noalign{\smallskip}											
RK547	&	0.500000	&	1.48E-01	&			&	4.41E-03	&		\\
		&	0.250000	&	7.31E-03	&	4.34	&	4.13E-05	&	6.74	\\
		&	0.125000	&	4.30E-04	&	4.09	&	3.29E-07	&	6.97	\\
		&	0.062500	&	2.65E-05	&	4.02	&	2.58E-09	&	6.99	\\
		&	0.031250	&	1.65E-06	&	4.01	&	2.02E-11	&	7.00	\\
		&	0.015625	&	1.03E-07	&	4.00	&	1.58E-13	&	7.00	\\
		\hline \noalign{\smallskip}											
	\end{tabular}	
	\end{center}									
\end{table}											
			
\subsection{Maxwell--Kerr system}	
	
In the final example, we consider the one-dimensional nonlinear Maxwell--Kerr system for modeling electromagnetic wave propagation in third-order nonlinear media:
\begin{align}
\varepsilon_0 \pd{D_z}{t} &= \pd{H_y}{x}, 						\label{eq:Kerr_D}	\\
		\mu_0 \pd{H_y}{t} &= \pd{E_z}{x},						\label{eq:Kerr_H}	\\
					  D_z &= \varepsilon_0 (E_z + \chi E_z^3). 	\label{eq:Kerr_DE}
\end{align}

Combining Equations (\ref{eq:Kerr_D}) and (\ref{eq:Kerr_DE}) yield
 \begin{equation}
 	\varepsilon_0 (1 + 3\chi E_z^2)\pd{E_z}{t} = \pd{H_y}{x}. \label{eq:Kerr_E}
 \end{equation}
Similar to the linear Maxwell example in the previous section, 
the nonlinear Maxwell--Kerr equations (\ref{eq:Kerr_H}) and (\ref{eq:Kerr_E}) are discretized on a spatially staggered grid using fourth-order centered differences:
\begin{align}
\frac{\partial}{\partial x} E\left(t_n,x_j+\frac{\Delta x}{2}\right) 	
	&\approx \frac{-E_{j+2}^n + 27E_{j+1}^n - 27E_{j}^n + E_{j-1}^n}{24\Delta x},	\label{eq:dEdx4} \\
\frac{\partial}{\partial x} H(t_n,x_j) 					
	&\approx \frac{-H_{j+3/2}^n + 27H_{j+1/2}^n - 27H_{j-1/2}^n + E_{j-3/2}^n}{24\Delta x},	\label{eq:dHdx4}	
\end{align} 
resulting in the following semi-discrete system:  
\beq \label{eq:Kerr}
\pdt \vtwo{\vec{E}}{\vec{H}} = \mtwo{\mathcal{A}/\varepsilon_0}{0}{0}{1/\mu_0} \mtwo{0}{C}{-C^\top}{0} \vtwo{\vec{E}}{\vec{H}},
\eeq
where $\vec{E}^n=(E^n_0,E^n_1,...,E^n_{N_x})^\top$ and 
$\vec{H}^n=(H^n_{1/2},H^n_{3/2},...,H^n_{N_x+1/2})^\top$ represent the solution vectors at time $n\Delta t$. 
The matrix $C$ represents the finite difference curl operator, and $\mathcal{A}$ is the diagonal matrix 
\begin{equation}
	\mathcal{A} = \operatorname{diag} \left( \frac{1}{1 + 3 \chi E_1^2}, \frac{1}{1 + 3 \chi E_2^2}, ..., \frac{1}{1 + 3 \chi E_{N_x}^2}\right).
\end{equation}

For example, when $N_x=4$ and periodic boundary conditions are applied, $C$ is a circulant matrix:
\begin{equation}
C=\frac{1}{24 \Delta x}\pmat{ 	
	{ 27}&{ -1}&{  0}&{  1}&{-27}\\
	{-27}&{ 27}&{ -1}&{  0}&{  1}\\
	{  1}&{-27}&{ 27}&{ -1}&{  0}\\
	{  0}&{  1}&{-27}&{ 27}&{ -1}\\
	{ -1}&{  0}&{  1}&{-27}&{ 27} }.
\end{equation}
The spectral radius of $C^\top C$ satisfies $\rho(C^\top C)\le \frac{49}{9\Delta x^2}$. Therefore, 
\beq
\left|\left| \mtwo{0}{C}{-C^\top}{0} \right|\right| = \sqrt{\mathcal{\rho}(C^\top C)} \le \frac{7}{3\Delta x}.
\eeq
The corresponding stability condition is given by $\Delta t \le \frac{3\lambda}{7c}\Delta x$, where $c$ is the speed of light and $\lambda$ is the stability limit provided in Table \ref{tab:p4}. 
Specifically, we set $\Delta t = \frac{6\sqrt{2}\Delta x}{7c}$ for RK4 and $\Delta t = \frac{6\sqrt{3}\Delta x}{7c}$ for RK547.

The total energy at $t=n\Delta t$ is given by
\begin{equation}
	\mathcal{E}_n = \frac{1}{2} \left(\mu_0 ||\vec{H}^n||^2 + \varepsilon_0 ||\vec{E}^n||^2 + \frac{3}{2} \varepsilon_0 \chi ||\vec{E}^n||^4 \right). 
\end{equation}

The computational domain is $x \in [-5, 5]$ with periodic boundary conditions. 
The final time is set to $T = 10^{-8}$ to ensure the pulse does not interact with the boundaries. 
The third-order electric susceptibility is set to $\chi = 0.0001$, and the initial condition is the Gaussian pulse defined in Equation (\ref{eq:initial_pulse}). 

The results for RK4 and RK547 are presented in Table \ref{tab:Kerr}. 
Notably, the RK547 method demonstrates a seventh-order energy convergence rate, outperforming RK4 in energy accuracy by two orders of magnitude. Furthermore, because RK547 utilizes a larger time step, its total simulation runtime remains comparable to that of RK4.

	Table \ref{tab:KerrEfficiency} compares the computational performance of RK4 and RK547 methods 
	required to achieve similar target energy accuracy. 
	Although both RK4 and RK547 are fourth-order accurate in the solution, RK547 possesses seventh-order energy accuracy, compared with fifth-order energy accuracy for RK4. 
	These results show that RK547 is more than 12 times faster than RK4 when achieving a target energy accuracy of $|\epsilon_E| < 5 \times 10^{-10}$, while it is about 65 times faster for a target accuracy of $|\epsilon_E| < 5 \times 10^{-13}$.
	These results indicate that the efficiency advantage of RK547 becomes increasingly significant as stricter energy accuracy is required.

\begin{table}																														
		\begin{center}															
			\caption{Energy convergence rates of Maxwell--Kerr system using the classical RK4 and the RK547 methods. } 															
			\label	{tab:Kerr}														
			\begin{tabular}	{lrrccccc}														
				\hline \noalign{\smallskip}															
		Method	&	$N_x$	&	$N_t$	& $\epsilon_2$ 	&	Rate	& $\epsilon_E$	&	Rate	&	Runtime (s) \\
				\hline \noalign{\smallskip}															
				&	1000	&	248		&	2.49E-03	&			&	-1.16E-02	&			&	0.15	\\
				&	2000	&	495		&	1.59E-04	&	3.97	&	-3.75E-04	&	4.95	&	0.20	\\	 		
		RK4		&	4000	&	990		&	9.98E-06	&	4.00	&	-1.18E-05	&	4.99	&	0.44	\\
				&	8000	&	1979	&	6.23E-07	&	4.00	&	-3.69E-07	&	4.99	&	1.17	\\
				&	16000	&	3958	&	3.66E-08	&	4.09	&	-1.15E-08	&	5.00	&	3.59	\\
				\hline \noalign{\smallskip}															
				&	1000	&	202		&	1.35E-03	&			&	-3.33E-04	&			&	0.14	\\
				&	2000	&	404		&	8.18E-05	&	4.05	&	-2.65E-06	&	6.98	&	0.22	\\
		RK547	&	4000	&	808		&	5.06E-06	&	4.01	&	-2.07E-08	&	7.00	&	0.45	\\
				&	8000	&	1616	&	3.14E-07	&	4.01	&	-1.54E-10	&	7.06	&	1.17	\\
				&	16000	&	3231	&	1.85E-08	&	4.09	&	-7.59E-13	&	7.67	&	3.49	\\
				\hline \noalign{\smallskip}															
			\end{tabular}															
		\end{center}																														
\end{table}

\begin{table}																									
		\begin{center}											
			\caption{Computational efficiency comparison of two fourth-order RK methods achieving similar target energy accuracy for Maxwell--Kerr system integrated up to $T=10^{-8}$.}									
			\label{tab:KerrEfficiency}										
			\begin{tabular}{lrrcc}										
				\hline \noalign{\smallskip}									
				method 	&	 $N_x$	&	$N_t$	&$\epsilon_E$	&	CPU time (s) 	\\ \hline\noalign{\smallskip}
				RK4		&	30,000	&	7,420	&	-4.99E-10	&	10.05	\\ \noalign{\smallskip}
				RK547	&	6,800	&	1,374	&	-4.91E-10	&	0.79	\\ \noalign{\smallskip}
				RK4		&	120,000	&	29,678	&	-4.87E-13	&	227.16	\\ \noalign{\smallskip}
				RK547	&	16,800	&	3,393	&	-4.81E-13	&	3.45	\\ \hline
			\end{tabular}										
		\end{center}																							
\end{table}


\section{Conclusion}
\label{sec:conc}
Through the study of energy convergence rates of explicit Runge--Kutta (RK) methods for skew-symmetric linear autonomous systems, we have developed a class of energy-superconvergent RK methods in which the orders of energy accuracy exceed the number of stages. For an $s$-stage RK method of an even order $p$, the energy accuracy can reach the order of $2s-p+1$. Generally, higher order energy accuracy leads to a larger stability region and thus a larger step size. Additionally, we have derived the strong stability criteria for several fourth-order methods when $5 \le s\le 7$. 
 The present strong stability result is related to classical strong stability preserving (SSP) theory: SSP methods are designed for general nonlinear systems, whereas our analysis exploits skew-symmetry systems, yielding high-order explicit RK methods with superconvergent energy accuracy.

Performance of methods with energy-accuracy up to the eleventh-order is demonstrated using harmonic oscillator ODEs, nonlocal integro-differential equations for peridynamics, and Maxwell's equations. 
In general, a method with higher-order energy accuracy exhibits better computational efficiency than lower-order ones. 
In particular, for the peridynamic simulations, the eleventh-order energy-accurate method is 2,400 times faster than RK4 
at reaching a similar energy accuracy of $5\times10^{-10}$.

Furthermore, we extend the three-stage energy-superconvergent RK method to 
autonomous nonlinear systems with amplitude-dependent frequencies, 
by deriving fifth-order energy conditions for three-stage, second-order methods, leading to RK325. 
The performance of RK325 is illustrated using several nonlinear examples, 
including Euler's equations for rigid body dynamics, the nonlinear Schr\"odinger equation, the KdV equation, Burgers' equation, and the Landau--Lifshitz equation. 
Notably, for Burgers' equation and the Landau--Lifshitz equation, RK325 outperforms RK4 by achieving a higher order of energy accuracy.

Next, we construct seventh-order energy-accurate methods for autonomous nonlinear systems with cubic nonlinearities.
Specifically, we obtain two four-stage, second-order methods (RK427a and RK427b) and one five-stage fourth-order method (RK547), all reaching seventh-order energy accuracy.
Additional simulations on the nonlinear Maxwell--Kerr system illustrate that the RK547 method is more than 65 times faster than RK4 at reaching a similar energy accuracy of $5\times10^{-13}$.

In this work, the coefficients for RK427 and RK547 are computed as numerical solutions to a large nonlinear system, tailored to autonomous nonlinear systems with amplitude-dependent frequencies and cubic nonlinearities. The proposed framework further establishes a foundation for deriving exact coefficients for more general nonlinear problems via B-series, facilitates the development of backward error analyses for these energy-superconvergent explicit RK methods, and supports systematic comparisons with symplectic and structure-preserving methods.






\section*{Appendix A: Order conditions for RK427}
\label{appendixA}
The following seven equations are the order conditions for RK427; the first two equations guarantee second-order accuracy for the solution, while the remainder ensure seventh-order energy accuracy.
\begin{align}
	&b_1 + b_2 + b_3 + b_4 = 1, \tag{A1} \\
	&b_2 c_2 + b_3 c_3 + b_4 c_4 = \frac{1}{2},  \tag{A2} \\
	&(b_3 a_{32} c_2 + b_4 a_{42} c_2 + b_4 a_{43} c_3) - b_4 a_{43} a_{32} c_2 = \frac{1}{8}, \tag{A3} \\
	&(b_2 c_2^2 + b_3 c_3^2 + b_4 c_4^2) - (b_2 c_2^3 + b_3 c_3^3 + b_4 c_4^3) \nonumber \\
	&\quad + 2(b_3 a_{32} c_2 c_3 + b_4 a_{42} c_2 c_4 + b_4 a_{43} c_3 c_4) \nonumber \\
	&\quad - (b_3 a_{32} c_2^2 + b_4 a_{42} c_2^2 + b_4 a_{43} c_3^2) = \frac{1}{4}, \tag{A4} \\
	&(b_3 a_{32} c_2 + b_4 a_{42} c_2 + b_4 a_{43} c_3)^2 - b_4 a_{43} a_{32} c_2 = 0, \tag{A5} \\	
	&(b_2 c_2^2 + b_3 c_3^2 + b_4 c_4^2)\left( \frac{3}{4} - 2 b_4 a_{43} a_{32} c_2 \right) 
		 - 2(b_3 a_{32} c_2^3 + b_4 a_{42} c_2^3 + a_{43} b_4 c_3^3) \nonumber \\
	&\quad + 6(b_3 a_{32}^2 c_2^2 + b_4 a_{42}^2 c_2^2 + b_4 a_{43}^2 c_3^2) 
		 - 2(b_3 a_{32} c_2 c_3^2 + b_4 a_{42} c_2 c_4^2 + b_4 a_{43} c_3 c_4^2) \nonumber \\
	&\quad - 2(b_3 a_{32}^2 c_2^2 c_3 + b_4 a_{42}^2 c_2^2 c_4 + b_4 a_{43}^2 c_3^2 c_4)
		 - 2 b_4 a_{43} a_{32} c_2 (3 a_{32} c_2 + 2 a_{42} c_2 + 2 a_{43} c_3) \nonumber \\
	&\quad + b_4 a_{43} a_{32} c_2 (2 c_2^2 + 2 c_3^2 + 6 c_4^2) 
		 + 2 b_4 a_{43} a_{32} c_2 (2 c_3 - 2 c_4 - c_2 + 1) \nonumber \\
	&\quad + 4 a_{42} a_{43} b_4 c_3 c_2 (3 - c_4) = \frac{1}{4}, \tag{A6} \\	
	&6(b_3 a_{32}^2 c_2^2 + b_4 a_{42}^2 c_2^2 + b_4 a_{43}^2 c_3^2) 
		 + 4(b_3 a_{32} c_2^2 c_3 + b_4 a_{42} c_2^2 c_4 + b_4 a_{43} c_3^2 c_4) \nonumber \\
	&\quad - 2(b_3 a_{32} c_2 c_3^2 + b_4 a_{42} c_2 c_4^2 + b_4 a_{43} c_3 c_4^2) 
	 - 6(b_3 a_{32} c_2^3 + b_4 a_{42} c_2^3 + b_4 a_{43} c_3^3) \nonumber \\
	&\quad - 6(b_3 a_{32} c_2 + b_4 a_{42} c_2 + b_4 a_{43} c_3)(b_2 c_2^2 + b_4 c_4^2 + b_3 c_3^2) \nonumber \\
	&\quad + (b_2 c_2^2 + b_3 c_3^2 + b_4 c_4^2)^2 + (b_2 c_2^3 + b_3 c_3^3 + b_4 c_4^3) \nonumber \\
	&\quad + b_3 a_{32} c_2^2 + b_4 a_{43} c_3^2 + b_4 a_{42} c_2^2 \nonumber \\
	&\quad - 4(b_3^2 a_{32} c_2 c_3^2 + b_4^2 a_{42} c_2 c_4^2 + b_4^2 a_{43} c_3 c_4^2) \nonumber \\
	&\quad + 4(b_3 a_{32}^2 c_2^3 + b_4 a_{42}^2 c_2^3 + b_4 a_{43}^2 c_3^3) \nonumber \\
	&\quad + 4(b_3 a_{32} c_2^3 c_3 + b_4 a_{42} c_2^3 c_4 + b_4 a_{43} c_3^3 c_4) \nonumber \\
	&\quad - 6(b_3 a_{32} c_2^2 c_3^2 + b_4 a_{42} c_2^2 c_4^2 + b_4 a_{43} c_3^2 c_4^2) \nonumber \\
	&\quad - 2(b_3 a_{32}^2 c_2^2 c_3 + b_4 a_{42}^2 c_2^2 c_4 + b_4 a_{43} c_3^2 c_4) \nonumber \\
	&\quad - 4(b_2 b_3 a_{32} c_2^2 c_3 + b_2 b_4 a_{42} c_2^2 c_4 + b_3 b_4 a_{43} c_3^2 c_4) \nonumber \\
	&\quad - 4(b_4 b_3 a_{32} c_2 c_3 c_4 + b_3 b_4 a_{42} c_2 c_3 c_4 + b_2 b_4 a_{43} c_2 c_3 c_4) \nonumber \\
	&\quad + 4 a_{42} b_4 a_{43} c_2 c_3 (c_2 + c_3 - c_4) \nonumber \\
	&\quad + 2 b_4 a_{43} a_{32} c_2 (3 c_2^2 + c_3^2 + 9 c_4^2 + 2 c_2 c_4 - 2 c_2 c_3 - 4 c_3 c_4) \nonumber \\
	&\quad - 6 b_4 a_{43} a_{32} c_2 (a_{32} c_2 + 2 a_{42} c_2 + 2 a_{43} c_3) \nonumber \\
	&\quad + 6 b_4 a_{43} a_{32} c_2 (2 c_3 - c_2 - 2 c_4 + 1) + 12 a_{42} a_{43} b_4 c_2 c_3 = 0. \tag{A7}
\end{align}

\section*{Appendix B: Order conditions for RK547}
\label{appendixB}
The following eleven equations are the order conditions for RK547.  
\begin{align}
	& b_1 + b_2 + b_3 + b_4 + b_5 = 1, \tag{B1} \\
	& b_2c_2 + b_3c_3 + b_4c_4 + b_5c_5 = \frac{1}{2}, \tag{B2} \\
	& b_2c_2^2 + b_3c_3^2 + b_4c_4^2 + b_5c_5^2 = \frac{1}{3}, \tag{B3} \\
	& b_3a_{32}c_2 + b_4a_{42}c_2 + b_4a_{43}c_3 + b_5a_{52}c_2 + b_5a_{53}c_3 + b_5a_{54}c_4 = \frac{1}{6}, \tag{B4} \\
	& b_2c_2^3 + b_3c_3^3 + b_4c_4^3 + b_5c_5^3 = \frac{1}{4}, \tag{B5} \\
	& b_3c_3a_{32}c_2 + b_4c_4a_{42}c_2 + b_4c_4a_{43}c_3 + b_5c_5a_{52}c_2 + b_5c_5a_{53}c_3 + b_5c_5a_{54}c_4 = \frac{1}{8}, \tag{B6} \\
	& b_3a_{32}c_2^2 + b_4a_{42}c_2^2 + b_4a_{43}c_3^2 + b_5a_{52}c_2^2 + b_5a_{53}c_3^2 + b_5a_{54}c_4^2 = \frac{1}{12}, \tag{B7} \\
	& b_4a_{43}a_{32}c_2 + b_5a_{53}a_{32}c_2 + b_5a_{54}a_{42}c_2 + b_5a_{54}a_{43}c_3 = \frac{1}{24}, \tag{B8} \\
	& a_{32}a_{43}a_{54}b_5c_2  = \frac{1}{144}, \tag{B9}
\end{align}
\begin{align}	
	& \frac{c_3}{6} - \frac{c_2}{2} - \frac{c_4}{3} + \frac{a_{32}c_2}{2} - \frac{a_{42}c_2}{3} - \frac{a_{43}c_3}{3} + \frac{c_2c_3}{3} + \frac{c_2c_4}{6} - \frac{c_3c_4}{3} \nonumber \\
	& + \frac{2a_{32}c_2^2}{3} - \frac{2b_3c_3^2}{3} + b_3c_3^4 - \frac{2b_4c_4^2}{3} - b_3c_3^5 + b_4c_4^4 - \frac{2b_5c_5^2}{3} - b_4c_4^5 \nonumber \\
	& + b_5c_5^4 - b_5c_5^5 - c_2c_3^2 + \frac{2c_2^2c_3}{3} - \frac{c_2^2}{4} + \frac{c_2^3}{3} - \frac{c_3^2}{3} - \frac{c_2^4}{2} + \frac{2c_3^3}{3} + \frac{c_4^2}{2} \nonumber \\
	& + b_3^2c_3^4 + b_4^2c_4^4 + b_5^2c_5^4 + 4a_{42}^2b_4c_2^2 + 4a_{42}^2b_4c_2^3 + 4a_{43}^2b_4c_3^2 + 4a_{43}^2b_4c_3^3 \nonumber \\
	& + 4a_{52}^2b_5c_2^2 + 4a_{52}^2b_5c_2^3 + 4a_{53}^2b_5c_3^2 + 4a_{53}^2b_5c_3^3 + 4a_{54}^2b_5c_4^2 + 4a_{54}^2b_5c_4^3 \nonumber \\
	& - 2b_3^2c_2c_3^3 - 2b_4^2c_2c_4^3 - 2b_5^2c_2c_5^3 - \frac{2a_{32}c_2c_3}{3} + \frac{2b_3c_2c_3}{3} + \frac{2b_4c_2c_4}{3} \nonumber \\
	& + \frac{2b_5c_2c_5}{3} + b_3^2c_2^2c_3^2 + b_4^2c_2^2c_4^2 + b_5^2c_2^2c_5^2 - 5a_{43}b_4c_3^4 - 5a_{53}b_5c_3^4 \nonumber \\
	& - 4a_{54}b_5c_4^3 - a_{54}b_5c_4^4 + b_3c_2c_3^2 - b_3c_2^2c_3 - b_3c_2^3c_3 + b_3c_2^4c_3 + b_4c_2c_4^2 - b_4c_2^2c_4 \nonumber \\
	& - b_4c_2^3c_4 + b_4c_2^4c_4 + b_5c_2c_5^2 - b_5c_2^2c_5 - b_5c_2^3c_5 + b_5c_2^4c_5 - 4a_{32}a_{42}b_4c_2^2 \nonumber \\
	& - 4a_{32}a_{42}b_4c_2^3 - 4a_{32}a_{52}b_5c_2^2 - 4a_{32}a_{52}b_5c_2^3 - 8a_{43}a_{54}b_5c_3^2 - 4a_{43}a_{54}b_5c_3^3 \nonumber \\
	& + 4a_{42}b_4c_2c_3^2 - 4a_{42}b_4c_2^2c_3 - 4a_{42}b_4c_2c_3^3 - 4a_{42}b_4c_2c_4^2 + 4a_{42}b_4c_2^2c_4 \nonumber \\
	& - 4a_{42}b_4c_2^3c_3 - 4a_{43}b_4c_2c_3^2 + 4a_{43}b_4c_2^2c_3 + 4a_{42}b_4c_2c_4^3 + 4a_{42}b_4c_2^3c_4 \nonumber 
\end{align}
\begin{align}	
	& + 6a_{43}b_4c_2c_3^3 + a_{43}b_4c_2^3c_3 - 4a_{43}b_4c_3c_4^2 + 4a_{43}b_4c_3^2c_4 + 4a_{43}b_4c_3c_4^3 \nonumber \\
	& + 4a_{43}b_4c_3^3c_4 + 4a_{52}b_5c_2c_3^2 - 4a_{52}b_5c_2^2c_3 - 4a_{52}b_5c_2c_3^3 - 4a_{52}b_5c_2^3c_3 \nonumber \\
	& - 4a_{53}b_5c_2c_3^2 + 4a_{53}b_5c_2^2c_3 - 4a_{52}b_5c_2c_5^2 + 4a_{52}b_5c_2^2c_5 + 6a_{53}b_5c_2c_3^3 \nonumber \\
	& + a_{53}b_5c_2^3c_3 + 4a_{52}b_5c_2c_5^3 + 4a_{52}b_5c_2^3c_5 + 4a_{54}b_5c_2^2c_4 - 4a_{53}b_5c_3c_5^2 \nonumber \\
	& + 4a_{53}b_5c_3^2c_5 + a_{54}b_5c_2^3c_4 + 4a_{54}b_5c_3^2c_4 + 4a_{53}b_5c_3c_5^3 + 4a_{53}b_5c_3^3c_5 \nonumber \\
	& - 4a_{54}b_5c_3^3c_4 - 4a_{54}b_5c_4c_5^2 + 4a_{54}b_5c_4^2c_5 + 4a_{54}b_5c_4c_5^3 + 4a_{54}b_5c_4^3c_5 \nonumber \\
	& + 4a_{42}^2a_{54}b_5c_2^2 + 4a_{43}^2a_{54}b_5c_3^2 + 6a_{42}b_4c_2^2c_3^2 - 6a_{42}b_4c_2^2c_4^2 - 4a_{43}b_4c_2^2c_3^2 \nonumber \\
	& - 4a_{42}^2b_4c_2^2c_4 - 6a_{43}b_4c_3^2c_4^2 - 4a_{43}^2b_4c_3^2c_4 + 6a_{52}b_5c_2^2c_3^2 - 4a_{53}b_5c_2^2c_3^2 \nonumber \\
	& - 6a_{52}b_5c_2^2c_5^2 - 4a_{52}^2b_5c_2^2c_5 - 6a_{53}b_5c_3^2c_5^2 - 4a_{53}^2b_5c_3^2c_5 - 6a_{54}b_5c_4^2c_5^2 \nonumber \\
	& - 4a_{54}^2b_5c_4^2c_5 + 2b_3b_4c_3^2c_4^2 + 2b_3b_5c_3^2c_5^2 + 2b_4b_5c_4^2c_5^2 + 8a_{32}a_{43}a_{54}b_5c_2 \nonumber \\
	& + 8a_{42}a_{43}b_4c_2c_3 + 8a_{32}a_{53}b_5c_2c_4 - 8a_{32}a_{53}b_5c_2c_5 - 4a_{32}a_{54}b_5c_2c_4 \nonumber \\
	& - 4a_{42}a_{54}b_5c_2c_3 + 16a_{42}a_{54}b_5c_2c_4 + 4a_{43}a_{54}b_5c_2c_3 - 8a_{42}a_{54}b_5c_2c_5 \nonumber \\
	& + 16a_{43}a_{54}b_5c_3c_4 - 8a_{43}a_{54}b_5c_3c_5 + 8a_{52}a_{53}b_5c_2c_3 + 8a_{52}a_{54}b_5c_2c_4 \nonumber \\
	& + 8a_{53}a_{54}b_5c_3c_4 - 4a_{54}b_5c_2c_3c_4 + 8a_{32}a_{42}a_{53}b_5c_2^2 + 4a_{32}a_{42}a_{54}b_5c_2^2 \nonumber \\
	& + 4a_{32}a_{43}a_{54}b_5c_2^2 - 8a_{32}a_{52}a_{53}b_5c_2^2 - 8a_{42}a_{52}a_{54}b_5c_2^2 - 8a_{43}a_{53}a_{54}b_5c_3^2 \nonumber \\
	& + 4a_{32}a_{42}b_4c_2^2c_3 + 4a_{32}a_{52}b_5c_2^2c_3 + 4a_{42}a_{43}b_4c_2c_3^2 + 4a_{42}a_{43}b_4c_2^2c_3 \nonumber \\
	& - 8a_{32}a_{53}^2b_5c_2c_3 - 12a_{32}a_{53}b_5c_2c_4^2 - 4a_{32}a_{53}b_5c_2^2c_4 + 12a_{32}a_{53}b_5c_2c_5^2 \nonumber \\
	& + 4a_{32}a_{53}b_5c_2^2c_5 - 4a_{32}a_{54}b_5c_2^2c_4 - 8a_{42}a_{54}b_5c_2c_3^2 + 8a_{42}a_{54}b_5c_2^2c_3 \nonumber \\
	& - 8a_{42}a_{54}b_5c_2c_4^2 - 8a_{42}a_{54}b_5c_2^2c_4 - 8a_{42}a_{54}^2b_5c_2c_4 + 8a_{43}a_{54}b_5c_2c_3^2 \nonumber \\
	& - 4a_{43}a_{54}b_5c_2^2c_3 + 12a_{42}a_{54}b_5c_2c_5^2 + 4a_{42}a_{54}b_5c_2^2c_5 - 8a_{43}a_{54}b_5c_3c_4^2 \nonumber \\
	& + 4a_{43}a_{54}b_5c_3^2c_4 - 8a_{43}a_{54}^2b_5c_3c_4 + 12a_{43}a_{54}b_5c_3c_5^2 + 4a_{43}a_{54}b_5c_3^2c_5 \nonumber \\
	& + 4a_{52}a_{53}b_5c_2c_3^2 + 4a_{52}a_{53}b_5c_2^2c_3 + 4a_{52}a_{54}b_5c_2c_4^2 + 4a_{52}a_{54}b_5c_2^2c_4 \nonumber \\
	& + 4a_{53}a_{54}b_5c_3c_4^2 + 4a_{53}a_{54}b_5c_3^2c_4 + 6a_{54}b_5c_2c_3^2c_4 - 4a_{54}b_5c_2^2c_3c_4 \nonumber \\
	& - 2b_3b_4c_2c_3c_4^2 - 2b_3b_4c_2c_3^2c_4 + 2b_3b_4c_2^2c_3c_4 - 2b_3b_5c_2c_3c_5^2 - 2b_3b_5c_2c_3^2c_5 \nonumber \\
	& + 2b_3b_5c_2^2c_3c_5 - 2b_4b_5c_2c_4c_5^2 - 2b_4b_5c_2c_4^2c_5 + 2b_4b_5c_2^2c_4c_5 + 8a_{32}a_{43}a_{53}b_5c_2c_3 \nonumber \\
	& - 4a_{32}a_{43}a_{54}b_5c_2c_3 + 8a_{32}a_{43}a_{54}b_5c_2c_4 - 8a_{32}a_{43}a_{54}b_5c_2c_5 + 8a_{42}a_{43}a_{54}b_5c_2c_3 \nonumber \\
	& - 8a_{32}a_{53}a_{54}b_5c_2c_4 - 8a_{42}a_{53}a_{54}b_5c_2c_3 - 8a_{43}a_{52}a_{54}b_5c_2c_3 - 8a_{42}a_{43}b_4c_2c_3c_4 \nonumber \\
	& + 8a_{32}a_{53}b_5c_2c_3c_4 - 8a_{32}a_{53}b_5c_2c_3c_5 + 4a_{32}a_{54}b_5c_2c_3c_4 + 8a_{42}a_{54}b_5c_2c_3c_4 \nonumber \\
	& - 4a_{43}a_{54}b_5c_2c_3c_4 - 8a_{42}a_{54}b_5c_2c_4c_5 - 8a_{43}a_{54}b_5c_3c_4c_5 - 8a_{52}a_{53}b_5c_2c_3c_5 \nonumber \\
	& - 8a_{52}a_{54}b_5c_2c_4c_5 - 8a_{53}a_{54}b_5c_3c_4c_5 + \frac{1}{9} = 0, \tag{B10} 
\end{align}
\begin{align}	
	& \frac{3a_{32}c_2}{4} - \frac{5c_3}{12} - \frac{c_4}{6} - \frac{c_2}{12} - \frac{a_{42}c_2}{6} - \frac{a_{43}c_3}{6} - \frac{b_3c_3^2}{3} + b_3c_3^3 - \frac{b_4c_4^2}{3} \nonumber \\
	& + b_4c_4^3 - \frac{b_5c_5^2}{3} + b_5c_5^3 + \frac{c_2^2}{4} - \frac{c_3^2}{6} + \frac{c_4^2}{4} + 6a_{42}^2b_4c_2^2 + 6a_{43}^2b_4c_3^2 + 6a_{52}^2b_5c_2^2 \nonumber \\
	& + 6a_{53}^2b_5c_3^2 + 6a_{54}^2b_5c_4^2 - \frac{a_{32}c_2c_3}{3} + \frac{b_3c_2c_3}{3} + \frac{b_4c_2c_4}{3} + \frac{b_5c_2c_5}{3} + 3a_{43}b_4c_3^2 \nonumber \\
	& + 3a_{53}b_5c_3^2 + a_{54}b_5c_4^2 - 2a_{54}b_5c_4^3 - b_3c_2^2c_3 - b_4c_2^2c_4 - b_5c_2^2c_5 - 6a_{32}a_{42}b_4c_2^2 \nonumber \\
	& - 6a_{32}a_{52}b_5c_2^2 - 2a_{43}a_{54}b_5c_3^3 + 2a_{42}b_4c_2c_3^2 - 2a_{42}b_4c_2c_4^2 + 2a_{43}b_4c_2^2c_3 \nonumber \\
	& - 2a_{43}b_4c_3c_4^2 + 2a_{52}b_5c_2c_3^2 + 2a_{53}b_5c_2^2c_3 - 2a_{52}b_5c_2c_5^2 + 2a_{54}b_5c_2^2c_4 \nonumber \\
	& - 2a_{53}b_5c_3c_5^2 + 2a_{54}b_5c_3^2c_4 - 2a_{54}b_5c_4c_5^2 - 2a_{42}^2a_{54}b_5c_2^2 - 2a_{43}^2a_{54}b_5c_3^2 \nonumber \\
	& - 2a_{42}^2b_4c_2^2c_4 - 2a_{43}^2b_4c_3^2c_4 - 2a_{52}^2b_5c_2^2c_5 - 2a_{53}^2b_5c_3^2c_5 - 2a_{54}^2b_5c_4^2c_5 \nonumber \\
	& + 2a_{42}b_4c_2c_3 - 2a_{42}b_4c_2c_4 - a_{43}b_4c_2c_3 - 2a_{43}b_4c_3c_4 + 2a_{52}b_5c_2c_3 \nonumber \\
	& - a_{53}b_5c_2c_3 - 2a_{52}b_5c_2c_5 - a_{54}b_5c_2c_4 - 2a_{53}b_5c_3c_5 + 2a_{54}b_5c_3c_4 \nonumber \\
	& - 2a_{54}b_5c_4c_5 + 8a_{32}a_{43}a_{54}b_5c_2 + 12a_{42}a_{43}b_4c_2c_3 + 4a_{32}a_{53}b_5c_2c_4 - 4a_{32}a_{53}b_5c_2c_5 \nonumber \\
	& - 6a_{32}a_{54}b_5c_2c_4 + 2a_{42}a_{54}b_5c_2c_3 + 8a_{42}a_{54}b_5c_2c_4 + 2a_{43}a_{54}b_5c_2c_3 \nonumber \\
	& - 4a_{42}a_{54}b_5c_2c_5 + 8a_{43}a_{54}b_5c_3c_4 - 4a_{43}a_{54}b_5c_3c_5 + 12a_{52}a_{53}b_5c_2c_3 \nonumber \\
	& + 12a_{52}a_{54}b_5c_2c_4 + 12a_{53}a_{54}b_5c_3c_4 + 4a_{32}a_{42}a_{53}b_5c_2^2 + 6a_{32}a_{42}a_{54}b_5c_2^2 \nonumber \\
	& + 2a_{32}a_{43}a_{54}b_5c_2^2 - 4a_{32}a_{52}a_{53}b_5c_2^2 - 4a_{42}a_{52}a_{54}b_5c_2^2 - 4a_{43}a_{53}a_{54}b_5c_3^2 \nonumber \\
	& + 2a_{32}a_{42}b_4c_2^2c_3 + 2a_{32}a_{52}b_5c_2^2c_3 - 4a_{32}a_{53}^2b_5c_2c_3 - 6a_{32}a_{53}b_5c_2c_4^2 \nonumber \\
	& + 6a_{32}a_{53}b_5c_2c_5^2 - 4a_{42}a_{54}b_5c_2c_3^2 - 4a_{42}a_{54}b_5c_2c_4^2 - 4a_{42}a_{54}^2b_5c_2c_4 \nonumber \\
	& - 2a_{43}a_{54}b_5c_2^2c_3 + 6a_{42}a_{54}b_5c_2c_5^2 - 4a_{43}a_{54}b_5c_3c_4^2 - 4a_{43}a_{54}^2b_5c_3c_4 \nonumber \\
	& + 6a_{43}a_{54}b_5c_3c_5^2 + 4a_{32}a_{43}a_{53}b_5c_2c_3 + 2a_{32}a_{43}a_{54}b_5c_2c_3 + 4a_{32}a_{43}a_{54}b_5c_2c_4 \nonumber \\
	& - 4a_{32}a_{43}a_{54}b_5c_2c_5 - 4a_{42}a_{43}a_{54}b_5c_2c_3 - 4a_{32}a_{53}a_{54}b_5c_2c_4 - 4a_{42}a_{53}a_{54}b_5c_2c_3 \nonumber \\
	& - 4a_{43}a_{52}a_{54}b_5c_2c_3 - 4a_{42}a_{43}b_4c_2c_3c_4 + 2a_{32}a_{54}b_5c_2c_3c_4 - 4a_{52}a_{53}b_5c_2c_3c_5 \nonumber \\
	& - 4a_{52}a_{54}b_5c_2c_4c_5 - 4a_{53}a_{54}b_5c_3c_4c_5 + \frac{1}{36} = 0. \tag{B11}
\end{align}

\bibliographystyle{unsrt}      
\bibliography{refs}   

@article{antoine2021scalar,
  title={Scalar auxiliary variable/{L}agrange multiplier based pseudospectral schemes for the dynamics of nonlinear {S}chr{\"o}dinger/{G}ross--{P}itaevskii equations},
  author={Antoine, Xavier and Shen, Jie and Tang, Qinglin},
  journal={Journal of Computational Physics},
  volume={437},
  pages={110328},
  year={2021},
  publisher={Elsevier}
}

@article{celledoni2010energy,
  title={Energy-preserving integrators and the structure of {B}-series},
  author={Celledoni, Elena and McLachlan, Robert I and Owren, Brynjulf and Quispel, Gilles Reinout Willem},
  journal={Foundations of Computational Mathematics},
  volume={10},
  number={6},
  pages={673--693},
  year={2010},
  publisher={Springer}
}

@article{cockburn2001runge,
	author = {Cockburn, Bernardo and Shu, Chi-Wang},
	journal = {Journal of Scientific Computing},
	pages = {173--261},
	publisher = {Springer},
	title = {{R}unge--{K}utta discontinuous {G}alerkin methods for convection-dominated problems},
	volume = {16},
	year = {2001}}

@article{coclite2020numerical,
	author = {Coclite, Giuseppe Maria and Fanizzi, Alessandro and Lopez, Luciano and Maddalena, Francesco and Pellegrino, Sabrina Francesca},
	journal = {Applied Numerical Mathematics},
	pages = {119--139},
	publisher = {Elsevier},
	title = {Numerical methods for the nonlocal wave equation of the peridynamics},
	volume = {155},
	year = {2020}}

@article{cooper1987stability,
  title={Stability of {R}unge--{K}utta methods for trajectory problems},
  author={Cooper, G.J.},
  journal={IMA journal of numerical analysis},
  volume={7},
  number={1},
  pages={1--13},
  year={1987},
  publisher={Oxford University Press}
}

@article{cui2025high,
  title={High-order and energy-preserving relaxed implicit-explicit {R}unge--{K}utta methods for {H}amiltonian {PDE}s},
  author={Cui, Lei and Huang, Qiong-Ao and Zhang, Gengen},
  journal={Mathematics and Computers in Simulation},
  year={2025},
  publisher={Elsevier}
}

@article{gottlieb1998total,
	author = {Gottlieb, Sigal and Shu, Chi-Wang},
	journal = {Mathematics of computation},
	number = {221},
	pages = {73--85},
	title = {Total variation diminishing {R}unge--{K}utta schemes},
	volume = {67},
	year = {1998}}

@article{gottlieb2001strong,
	author = {Gottlieb, Sigal and Shu, Chi-Wang and Tadmor, Eitan},
	journal = {SIAM review},
	number = {1},
	pages = {89--112},
	publisher = {SIAM},
	title = {Strong stability-preserving high-order time discretization methods},
	volume = {43},
	year = {2001}}

@book{guckenheimer2013nonlinear,
  title={Nonlinear oscillations, dynamical systems, and bifurcations of vector fields},
  author={Guckenheimer, John and Holmes, Philip},
  year={2013},
  publisher={Springer Science \& Business Media}
}

@article{hairer2003geometric,
	author = {Hairer, Ernst and Lubich, Christian and Wanner, Gerhard},
	journal = {Acta numerica},
	pages = {399--450},
	publisher = {Cambridge University Press},
	title = {Geometric numerical integration illustrated by the {S}t{\"o}rmer--{V}erlet method},
	volume = {12},
	year = {2003}}

@article{hairer2010energy,
  title={Energy-preserving variant of collocation methods},
  author={Hairer, E},
  journal={JNAIAM},
  volume={5},
  number={1-2},
  pages={73--84},
  year={2010}
}

@article{ketcheson2019relaxation,
  title={Relaxation {R}unge--{K}utta methods: Conservation and stability for inner-product norms},
  author={Ketcheson, David I},
  journal={SIAM Journal on Numerical Analysis},
  volume={57},
  number={6},
  pages={2850--2870},
  year={2019},
  publisher={SIAM}
}

@article{levy1998semidiscrete,
	author = {Levy, Doron and Tadmor, Eitan},
	journal = {SIAM review},
	number = {1},
	pages = {40--73},
	publisher = {SIAM},
	title = {From semidiscrete to fully discrete: Stability of {R}unge--{K}utta schemes by the energy method},
	volume = {40},
	year = {1998}}

@article{liu2024iterated,
	author = {Liu, Jinjie and Appiah-Adjei, Samuel and Brio, Moysey},
	journal = {Dynamics},
	number = {1},
	pages = {192--207},
	publisher = {MDPI},
	title = {Iterated {C}rank--{N}icolson Method for Peridynamic Models},
	volume = {4},
	year = {2024}}

@article{liu2024iterated2,
  title={Iterated {C}rank--{N}icolson {R}unge--{K}utta Methods and Their Application to {W}ilson--{C}owan Equations and Electroencephalography Simulations},
  author={Liu, Jinjie and Lu, Qi and Boukari, Hacene and Boukari, Fatima},
  journal={Foundations},
  volume={4},
  number={4},
  pages={673--689},
  year={2024},
  publisher={MDPI}
}

@article{quispel2008new,
  title={A new class of energy-preserving numerical integration methods},
  author={Quispel, Gilles Reinout Willem and McLaren, David Ian},
  journal={Journal of Physics A: Mathematical and Theoretical},
  volume={41},
  number={4},
  pages={045206},
  year={2008}
}

@article{ranocha2020relaxation,
  title={Relaxation {R}unge--{K}utta methods for {H}amiltonian problems},
  author={Ranocha, Hendrik and Ketcheson, David I},
  journal={Journal of Scientific Computing},
  volume={84},
  number={1},
  pages={17},
  year={2020},
  publisher={Springer}
}

@article{sanz1988runge,
  title={{R}unge--{K}utta schemes for {H}amiltonian systems},
  author={Sanz-Serna, Jesus M},
  journal={BIT Numerical Mathematics},
  volume={28},
  number={4},
  pages={877--883},
  year={1988},
  publisher={Springer}
}

@article{shen2018scalar,
  title={The scalar auxiliary variable ({SAV}) approach for gradient flows},
  author={Shen, Jie and Xu, Jie and Yang, Jiang},
  journal={Journal of Computational Physics},
  volume={353},
  pages={407--416},
  year={2018},
  publisher={Elsevier}
}

@article{shin2022energy,
  title={Energy conserving successive multi-stage method for the linear wave equation},
  author={Shin, Jaemin and Lee, June-Yub},
  journal={Journal of Computational Physics},
  volume={458},
  pages={111098},
  year={2022},
  publisher={Elsevier}
}

@article{shin2023energy,
  title={Energy-conserving successive multi-stage method for the linear wave equation with forcing terms},
  author={Shin, Jaemin and Lee, June-Yub},
  journal={Journal of Computational Physics},
  volume={489},
  pages={112255},
  year={2023},
  publisher={Elsevier}
}

@article{silling2000,
	author = {Silling, Stewart A},
	journal = {Journal of the Mechanics and Physics of Solids},
	number = {1},
	pages = {175--209},
	publisher = {Elsevier},
	title = {Reformulation of elasticity theory for discontinuities and long-range forces},
	volume = {48},
	year = {2000}}

@article{sun2017stability,
	author = {Sun, Zheng and Shu, Chi-Wang},
	journal = {Annals of Mathematical Sciences and Applications},
	number = {2},
	pages = {255--284},
	publisher = {International Press of Boston},
	title = {Stability of the fourth order {R}unge--{K}utta method for time-dependent partial differential equations},
	volume = {2},
	year = {2017}}

@article{sun2019strong,
	author = {Sun, Zheng and Shu, Chi-Wang},
	journal = {SIAM Journal on Numerical Analysis},
	number = {3},
	pages = {1158--1182},
	publisher = {SIAM},
	title = {Strong stability of explicit {R}unge--{K}utta time discretizations},
	volume = {57},
	year = {2019}}

@article{tadmor2002semidiscrete,
	author = {Tadmor, Eitan},
	journal = {Collected lectures on the preservation of stability under discretization},
	pages = {25--49},
	publisher = {Citeseer},
	title = {From semidiscrete to fully discrete: stability of {R}unge--{K}utta schemes by the energy method. II},
	volume = {109},
	year = {2002}}

@article{Taflove75,
	author = {A. Taflove and M. E. Brodwin},
	journal = {IEEE Trans. Microwave Theory Tech.},
	pages = {623--630},
	title = {Numerical Solution of Steady-State Electromagnetic Scattering Problems Using the Time-Dependent {M}axwell'S Equations},
	volume = {23},
	year = {1975}}

@book{Taflove05,
	address = {Norwood, MA},
	author = {A. Taflove and S. Hagness},
	edition = {3rd},
	publisher = {Artech House},
	title = {Computational Electrodynamics: The Finite-Difference Time-Domain Method},
	year = {2005}}

@article{Teukolsky00,
	author = {Teukolsky, Saul A},
	journal = {Physical Review D},
	number = {8},
	pages = {087501},
	publisher = {APS},
	title = {Stability of the iterated {C}rank--{N}icholson method in numerical relativity},
	volume = {61},
	year = {2000}}

@article{weckner2005effect,
	author = {Weckner, Olaf and Abeyaratne, Rohan},
	journal = {Journal of the Mechanics and Physics of Solids},
	number = {3},
	pages = {705--728},
	publisher = {Elsevier},
	title = {The effect of long-range forces on the dynamics of a bar},
	volume = {53},
	year = {2005}}

@article{yang2022energy,
  title={Energy conserving discontinuous {G}alerkin method with scalar auxiliary variable technique for the nonlinear {D}irac equation},
  author={Yang, Ruize and Xing, Yulong},
  journal={Journal of Computational Physics},
  volume={463},
  pages={111278},
  year={2022},
  publisher={Elsevier}
}

@article{Yee66,
	author = {Yee, Kane},
	journal = {IEEE Transactions on antennas and propagation},
	number = {3},
	pages = {302--307},
	publisher = {IEEE},
	title = {Numerical solution of initial boundary value problems involving {M}axwell's equations in isotropic media},
	volume = {14},
	year = {1966}}

\end{document}